\documentclass[12pt,a4paper]{amsart}
\usepackage{amsmath, amssymb, amsfonts,enumerate}
\usepackage[all]{xy}
\usepackage{amscd}
\usepackage{comment}

\newtheorem{theorem}{Theorem}[section]
\newtheorem{lemma}[theorem]{Lemma}
\newtheorem{corollary}[theorem]{Corollary}
\newtheorem{proposition}[theorem]{Proposition}

 \theoremstyle{definition}
 \newtheorem{definition}[theorem]{Definition}
 \newtheorem{remark}[theorem]{Remark}

 \newtheorem{example}[theorem]{Example}
\newtheorem{examples}[theorem]{Examples}

\numberwithin{equation}{section}
\newcommand {\N}{\mathbb{N}} %% positive integers
\newcommand {\Z}{\mathbb{Z}} %% integers
\newcommand {\R}{\mathbb{R}} %% reals
\newcommand{\BB}{\mathcal{B}}
\newcommand{\CC}{\mathcal{C}}

\newcommand{\OO}{\mathcal{O}}
\newcommand{\PP}{\mathcal{P}}

\newcommand{\UU}{\mathcal{U}}

\DeclareMathOperator{\Per}{Per}

\DeclareMathOperator{\Inv}{Inv}
\DeclareMathOperator{\CA}{CA}

\DeclareMathOperator{\End}{End}

\DeclareMathOperator{\Id}{Id}

\DeclareMathOperator{\Map}{Map}

\begin{document}
\title{On surjunctive monoids}
\author{Tullio Ceccherini-Silberstein}
\address{Dipartimento di Ingegneria, Universit\`a del Sannio, C.so
Garibaldi 107, 82100 Benevento, Italy}
\email{tceccher@mat.uniroma3.it}
\author{Michel Coornaert}
\address{Institut de Recherche Math\'ematique Avanc\'ee,
UMR 7501,                                             Universit\'e  de Strasbourg et CNRS,
                                                 7 rue Ren\'e-Descartes,
                                               67000 Strasbourg, France}
\email{coornaert@math.unistra.fr}
\begin{abstract}
A monoid $M$ is called surjunctive if every injective cellular automata with finite alphabet over $M$ is surjective.
We show that all finite monoids, all finitely generated commutative monoids, all cancellative commutative monoids, all residually finite monoids, all finitely generated linear monoids, 
and all cancellative one-sided amenable monoids are surjunctive.
We also prove that every limit of marked surjunctive   monoids is itself surjunctive.
On the other hand, we show that the bicyclic monoid and, more generally, all monoids containing a submonoid isomorphic to the bicyclic monoid are non-surjunctive.  
\end{abstract}
\date{\today}
\maketitle

%\tableofcontents
% SECTION 1
\section{Introduction}

The mathematical theory of cellular automata emerged from the pioneering work of John von Neumann in the 1940s. 
Von Neumann considered cellular automata over the additive groups $\Z^d$ and was interested in finding models for self-reproducing machines.
Subsequently, the problem of characterizing surjective and reversible cellular automata attracted considerable attention in particular because of its importance  
  when cellular automata are used  for  modeling time-evolving systems  in thermodynamics and many other branches of natural sciences.
A major breakthrough in that direction was the proof in 1963 by Moore~\cite{moore} and Myhill~\cite{myhill} of the so-called Garden of Eden theorem. 
The Garden of Eden theorem states that a cellular automaton with finite alphabet over $\Z^d$ is surjective if and only if it admits no mutually erasable patterns.
The absence of mutually erasable patterns is a weak form of injectivity.
Therefore, the Garden of Eden theorem implies in particular that every injective  cellular automaton with finite alphabet over $\Z^d$ is surjective. 
In other words, using a terminology coined by Gottschalk in~\cite{gottschalk}, 
the groups $\Z^d$ are \emph{surjunctive}. 
  From the 1960s,
it was progressively realized that it could be interesting to consider cellular automata
over groups other than the finitely generated free-abelian groups $\Z^d$
and also that the theory of cellular automata has strong connections with symbolic dynamics.  
 The question whether every group is surjunctive was raised by Gottschalk~\cite{gottschalk}.
Although this question, which is now known as the \emph{Gottschalk conjecture},
  remains open  in its full generality, the answer has been shown to be affirmative for a large class of groups.
Firstly, Lawson (see~\cite{gottschalk}) proved that every residually finite group is surjunctive.
Then,   it was shown in~\cite{ceccherini} that the Garden of Eden theorem remains true for all amenable groups, so that every amenable group is surjunctive.
Finally, Gromov~\cite{gromov-esav} and Weiss~\cite{weiss-sgds} introduced the class of sofic groups, a very large class of groups containing in particular all residually finite groups and all amenable groups, and proved that every sofic group is surjunctive.
It turns out that all groups that are known to be surjunctive are in fact sofic.
Actually, the question of the existence of a non-sofic group remains also  open up to now.
  \par
Our goal in this paper is to investigate the notion of surjunctivity for monoids, i.e., sets with an associative binary operation and an identity element.
In contrast with the group case, it is not difficult  to give examples of non-surjunctive monoids:
every monoid containing an element that is invertible from one side but not from the other
  is non-surjunctive.    
This may be rephrased by saying that  every monoid containing a submonoid isomorphic to the bicyclic monoid is non-surjunctive
(cf.~Proposition~\ref{p:charact-bicyclic-supmonoid} and Theorem~\ref{t:contains-bicyclic-is-non-surj}).
On the other hand, we shall see that plenty of monoids are surjunctive.
Indeed, the class of surjunctive monoids contains
all finite monoids, all finitely generated commutative monoids, all cancellative commutative monoids, all residually finite monoids, all finitely generated linear monoids, 
and all cancellative one-sided amenable monoids.
The surjunctivity of cancellative one-sided amenable monoids and the non-surjunctivity of the bicyclic monoid had been already observed in~\cite{semieden}. 
 \par
      The paper is organized as follows.
The first two sections introduce  notation and collect basic facts about monoids and uniform spaces.
In Section~\ref{s:cellular}, we  present a detailed exposition of the general theory of cellular automata over monoids
 (with finite or infinite alphabet)
analogous to the one for groups contained in~\cite{livre}.
 Although the extension from groups to monoids of most of the results of the theory  is straightforward,
the absence of inverses in monoids is responsible for the appearance of new phenomena.
For example, when considering cellular automata over monoids, it may happen that the restriction of an injective  cellular automaton is non-injective and that the cellular automaton induced by a surjective cellular automaton is non-surjective (see Example~\ref{ex:contre-ex-ind-res}).
The results about surjunctive  monoids mentioned above are established in Section~\ref{sec:surj-monoids}.
We also prove   that the class of surjunctive monoids is closed under taking submonoids (Theorem~\ref{t:sub-surjunctive}).
The proof of the surjunctivity of residually finite monoids presented in this section (Theorem~\ref{t:res-finite-surj}) is based on the density of periodic configurations in shifts over residually finite monoids.
In Section~\ref{sec:marked-monoids},
we introduce the space of marked monoids, obtained as quotients of a fixed base monoid,
and show that a limit of surjunctive marked monoids is itself surjunctive (Theorem~\ref{t:surjunct-closed}).
This yields another proof of the surjunctivity of residually finite monoids.
A list of open problems is given in the final section.

% SECTION 2
\section{Background material on monoids}

We have collected in this section   some basic facts about monoids that will be needed in the sequel.
 For a detailed exposition of the general theory of monoids,
the reader is invited to consult for example \cite{clifford-preston}.

\subsection{Monoids}
A \emph{monoid} is a set  equipped with an associative binary operation and admitting an identity element. 
\par
Let $M$ be a monoid.
 We use a multiplicative notation for the binary operation on $M$ and denote by $1_M$ its identity element.
 \par
 The \emph{opposite monoid} of   $M$ is the monoid with underlying set $M$ and monoid operation $*$ given by
$m * m':= m' m$ for all $m,m' \in M$.
\par
Given $m \in M$, we denote by $L_m$ and $R_m$ the left and right multiplication by $m$, that is,
the maps $L_m \colon M \to M$ and $R_m \colon M \to M$ defined by $L_m(m') = mm'$ and $R_m(m') = m'm$ for all $m' \in M$.
An element $m \in M$  is called \emph{left-cancellable} (resp. \emph{right-cancellable}) 
if the map $L_m$ (resp. $R_m$) is injective.
One says that an element $m \in M$ is cancellable if it is both left-cancellable and right-cancellable.
The monoid $M$ is called \emph{left-cancellative} (resp. \emph{right-cancellative}, resp. \emph{cancellative}) if every element in $M$ is left-cancellable  (resp. right-cancellable, 
resp. cancellable).
An element $m \in M$ is called
\emph{right-invertible} (resp.~\emph{left-invertible}) if  the map  $L_m$ (resp.~$R_m$)
is surjective.
This amounts to saying that there exists an element $m' \in M$ such that $m m' = 1_M$
(resp.~$m' m = 1_M$).
Such an element $m'$ is then called a \emph{right-inverse}
(resp.~\emph{left-inverse}) of $m$.
An element  $m \in M$ is called \emph{invertible} if it is both right-invertible  and left-invertible.
This amounts to saying that there exists an element $m' \in M$ such that $m m' = m' m = 1_M$.
Such an element $m'$ is then unique and is called the \emph{inverse} of $m$, because
it is both the unique right-inverse of $m$ and its unique left-inverse. 
    A \emph{group} is a monoid in which every element is invertible.
\par
A \emph{submonoid} of $M$ is a subset $N \subset M$ such that $1_M \in N$ and $m m' \in N$ for all $m, m' \in N$.
If $N \subset M$ is a submonoid, then $N$ inherits from $M$ a monoid structure obtained by restricting to $N$ the monoid operation on $M$.
\par
Given a subset $S \subset M$,  the submonoid \emph{generated} by $S$, denoted $\langle S \rangle$, is the smallest submonoid of $M$ containing $S$.
 It is the intersection of all the  submonoids of $M$ containing $S$ and consists of all elements   of the form $s_1s_2 \cdots s_n$ where $n \geq 0$ and $s_i \in S$ for all $1 \leq i \leq n$.
One says that the monoid $M$ is \emph{finitely generated} if there exists a finite subset $S \subset M$ such that $M = \langle S \rangle$.
\par
Given   monoids $M$ and $N$,
a map   $\varphi \colon M \to N$ 
is called a \emph{monoid morphism}   
if it satisfies $\varphi(1_{M}) = 1_{N}$ and $\varphi(m m') = \varphi(m) \varphi(m')$ for all $m,m' \in M$.  
A bijective monoid morphism is called a \emph{monoid isomorphism}.
The monoids $M$ and $N$ are said to be \emph{isomorphic} if there exists a monoid isomorphism
$\varphi \colon M \to N$.
One says that a monoid $M$ \emph{embeds} in a monoid $N$ if $M$ is isomorphic to a submonoid of $N$.

\begin{example}
The \emph{symmetric monoid} of a set $X$
is the monoid $\Map(X)$ consisting of all maps
$f \colon X \to X$ with the composition of maps as the monoid operation.
The identity element of   $\Map(X)$ is the identity map $\Id_X \colon X \to X$.
\par
Note that every monoid $M$ embeds in the symmetric monoid $\Map(M)$.
Indeed, the \emph{Cayley map}, that is, the map sending each element $m \in M$ to the map $L_m \in \Map(M)$, yields an injective monoid morphism
from $M$ into $\Map(M)$.
\end{example}

\begin{example}
The \emph{free monoid} based on a set $X$ is the monoid
$X^\star$ consisting of all words with finite length over the alphabet $X$ with the concatenation of words as the monoid operation.
The identity element of   $X^\star$ is the empty word $\varepsilon \in X^\star$, i.e., the only word with length $0$.
Note that if $X$ is reduced to a single element
then $X^\star$ is isomorphic to the additive monoid $\N$ of non-negative integers.
\end{example}

\subsection{Congruence relations} 
Let $X$ be a set. We denote by $\Delta_X:= \{(x,x): x \in X\} \subset X \times X$ the
\emph{diagonal} in $X \times X$.
Given an equivalence relation $\rho \subset X \times X$ on $X$ and an element $x \in X$, we denote by 
$[x] := \{y \in X: (x,y) \in \rho\} \subset X$ the $\rho$-\emph{class} of $x$.
The quotient of $X$ by $\rho$ is the set    $X/\rho := \{[x] : x \in X\}$ 
consisting  of all $\rho$-classes. We say that $\rho$ is of \emph{finite index}
if the set $X/\rho$ is finite. 
 Note that if $\rho, \rho' \subset X \times X$ are two
equivalence relations on $X$ such that $\rho \subset \rho'$ and $\rho$ is of finite index, then 
$\rho'$ is also of finite index and one has $|X/\rho'| \leq |X/\rho|$.
\par
Let $M$ be a monoid.
A \emph{congruence relation} on $M$ is an equivalence relation  $\gamma \subset M \times M$ on $M$ 
 satisfying  
\begin{equation}
\label{e:congruence}
(m_1,m'_1), (m_2,m'_2) \in \gamma \Rightarrow (m_1m_2, m'_1m'_2) \in \gamma
\end{equation}
for all $m_1,m_2,m'_1, m'_2 \in M$.
\par
If $\varphi \colon M \to N$ is a monoid morphism,   then the set
\begin{equation}
\label{e:cong-rel}
\gamma_\varphi := \{(m,m'): \varphi(m) = \varphi(m')\} \subset M \times M
\end{equation}
is a congruence relation on $M$.
It is  called the \emph{kernel congruence relation} of $\varphi$.
Note that $\varphi$ is injective if and only if $\gamma_\varphi=\Delta_M$.
\par
Given a congruence relation $\gamma$ on $M$, the quotient set $M/\gamma$   inherits a natural monoid structure  with the multiplication given by
\begin{equation}
\label{e:congruence-multiplication}
[m][m']:=[mm']
\end{equation}
for all $m,m' \in M$.   
This monoid structure on $M/\gamma$ is the only one for which the canonical surjective map $\pi_\gamma \colon M \to M/\gamma$, defined by $\pi_\gamma(m) := [m]$ for all $m \in M$, 
 is a monoid morphism.
 Moreover, $\gamma$ is the kernel congruence relation of $\pi_\gamma$.
  \par 
 Conversely, suppose that $\varphi \colon M \to N$ is a surjective monoid morphism
 and let $\gamma$ denote the kernel congruence relation of $\varphi$. 
 Then the monoid $M/\gamma$ is canonically isomorphic to the monoid $N$ via the map $[m] \mapsto \varphi(m)$.
Moreover,  after identifying $M/\gamma$ and $N$ by means of  this isomorphism, 
one has  $\pi_\gamma = \varphi$. 
\par
Given a subset $R \subset M \times M$, one defines the congruence relation \emph{generated} by $R$,
denoted $\gamma(R)$, as the smallest  congruence relation containing    $R$. 
It is the intersection of all congruence relations on $M$ that contain  $R$.
\par
Consider the free monoid  $X^\star$  based on a set $X$.
Given a subset $R \subset X^* \times X^*$, let $\gamma(R) \subset X^* \times X^*$ denote the congruence relation on    $X^*$ generated by $R$ and let $M:=X^*/\gamma(R)$ be the corresponding quotient monoid. Then the expression $\langle X; \{r = r': (r,r') \in R\}\rangle$ is called the \emph{presentation} of $M$ with \emph{generators} $X$ 
and \emph{relators} $R$.
\begin{comment}

\begin{lemma}
\label{l:cancellative}
Let $M$ be a monoid and $\gamma \subset M \times M$ a congruence relation on $M$.
The following conditions are equivalent:
\begin{enumerate}[{\rm (a)}]
\item for all $m,m', m'' \in M$, one has 
\[
(mm',mm'') \in \gamma \Longrightarrow (m',m'') \in \gamma;  
\]
\item 
the quotient monoid $M/\gamma$ is left-cancellative.
\end{enumerate}
\end{lemma}
\begin{proof}
Let $m,m',m'' \in M$. Suppose that $M$ satisfies (a) and that $[m][m'] = [m][m'']$.
Then $(mm',mm'') \in \gamma$ and therefore $(m',m'') \in \gamma$, equivalently
$[m'] = [m'']$, showing that $M/\gamma$ is left-cancellative. This proves (a) $\Rightarrow$ (b).
Conversely, suppose that $M/\gamma$ is left-cancellative and that $(mm',mm'') \in \gamma$.
Then $[m][m'] = [mm'] = [mm''] = [m][m'']$ and by left-cancellability $[m'] = [m'']$,
equivalently $(m',m'') \in \gamma$. This proves (b) $\Rightarrow$ (a).
\end{proof}
We shall say that a congruence relation $\gamma \subset M \times M$ is \emph{left-cancellative} provided it satisfies one of the two equivalent conditions in Lemma \ref{
l:cancellative}.
\end{comment}

\subsection{The bicyclic monoid}

The \emph{bicyclic monoid} is the monoid $B$ given by the presentation 
$B = \langle x,y : x y = 1 \rangle$.
The canonical images of $x$ and $y$ in $B$ are denoted by $p$ and $q$ respectively. 
Every element $m \in B$ may be uniquely written in the form
$m = q^a p^b$, where $a = a(m)$ and $b = b(m)$ are non-negative integers.
\par
The bicyclic monoid may also be viewed as a submonoid of the symmetric monoid $\Map(\N)$
 by regarding $p$ and $q$ as the maps from $\N$ into itself respectively defined by
\[
p(n) =
\begin{cases}
n - 1 &\text{ if } n \geq 1 \\
0 & \text{ if } n = 0
\end{cases}
\quad
\text{and}
\quad
q(n) = n + 1
\quad
\text{for all }  n \in \N.
\] 

 We have the following very useful characterization of monoids containing a submonoid isomorphic to the bicyclic monoid:

\begin{proposition}
\label{p:charact-bicyclic-supmonoid}
Let $M$ be a monoid.
Then the following conditions are equivalent:
\begin{enumerate}[\rm (a)]
\item
$M$ contains a submonoid isomorphic to the bicyclic monoid;
\item
$M$ contains an element that is left-invertible but not right-invertible;
\item
$M$ contains an element that is right-invertible but not left-invertible. 
\end{enumerate}
\end{proposition}

\begin{proof}
Suppose that $M$ contains a submonoid isomorphic to the bicyclic monoid.
Then there exist elements $p,q \in M$ such that
$pq = 1_M \not= qp$.
The element $p$ is right-invertible. However, $p$  is not left-invertible since otherwise it would be invertible with inverse $q$.
Similarly, $q$ is left-invertible but not right-invertible.
This shows that (a) implies (b) and (c).
\par
Suppose now that (b) or (c) is satisfied.
 This means  that there exist elements $s, t \in M$ such that $s t = 1_M \not= t s$. 
We claim that the submonoid of $M$ generated by $s$ is infinite.
Indeed, suppose that
\begin{equation}
\label{e:if-s-generates-finite}
s^u = s^v
\end{equation}
for some integers $v > u \geq 0$.
After multiplying both sides of~\eqref{e:if-s-generates-finite} by $t^u$ on the right and using the relation $s t = 1_M$, we would get
$$
s^{v - u} = 1_M
$$
which is impossible, since
it would imply that $s$ is invertible (with inverse $s^{v - u - 1}$). 
 This proves our claim. 
\par
Now, since $s t = 1_M$, there is a monoid morphism 
$\varphi \colon B \to M$ such that $\varphi(p) = s$ and $\varphi(q) = t$.
Let us show that $\varphi$ is injective.
Suppose that  $m, m' \in B$ satisfy $\varphi(m) = \varphi(m')$.
Write $m = q^a p^b$ and $m' = q^{a'} p^{b'}$, with $a,b,a',b' \in \N$.
We can assume $a \leq a'$.
As $\varphi(m) = \varphi(m')$ and $\varphi$ is a monoid morphism, we have that
\begin{equation}
\label{e:proof-bicyc-emb}
t^a s^b = t^{a'} s^{b'}.
\end{equation}
After multiplying both sides of this equality by $s^{a'}$ on the left and using $s t = 1_M$, this gives us
$$
 s^{a'-a+b} = s^{b'}.
$$
Since, as observed above, the submonoid generated by $s$ is infinite, 
this implies $a' - a + b = b'$.
Putting $d := a - a' = b - b' \geq 0$, equality \eqref{e:proof-bicyc-emb} becomes
$$
t^a s^b = t^{a + d} s^{b + d}
$$
which yields, after multiplying both sides by $s^a$ on the left and by $t^b$ on the right
\begin{equation}
\label{e:proof-bicyc-emb-4}
1_M = t^d s^d.
\end{equation}
It follows that $d = 0$ since, otherwise, we would get $1_M = t s$ by multiplying both sides of \eqref{e:proof-bicyc-emb-4} by $s^{d - 1}$ on the left and $t^{d - 1}$ on the right.
We conclude  that $m = m'$.
This shows that $\varphi$ is injective.
Consequently, the monoid generated by $s$ and $t$ is isomorphic to the bicyclic monoid $B$.
It follows that $M$ satisfies (a).   
\end{proof}

\subsection{Actions of monoids}
Let $M$ be a monoid.
\par
 An \emph{action} of $M$ on a set $X$ is a map from $ M \times X$ to $X$,
 denoted by $(m,x) \mapsto m x$,
satisfying $ 1_M x=x$ and $m_1(m_2 x) = (m_1 m_2) x$   for all $x \in X$ and $m_1,m_2 \in M$.
 These two conditions  amount to saying that the map
 from $M$ into $\Map(X)$,
obtained by  sending each $m \in M$ to the map $x \mapsto mx$,
 is a monoid morphism.
 A set equipped with an action of  $M$ is also called an $M$-\emph{set}.

\begin{example}
Let $A$ be a set, called the \emph{alphabet} or the \emph{set of symbols}. Consider the set $X := A^M$ consisting of all maps $x \colon M \to A$.
The elements of $X$ are called the \emph{configurations} over the monoid $M$ and the alphabet $A$. 
There is a natural  action of $M$  on $X$ defined by
$(m,x) \mapsto mx$, where
 $$
m x := x \circ R_m
$$
for all $m \in M$ and $x \in X$. 
This formula means that the configuration $m x$ is given by 
\begin{equation*}
mx(m') = x(m' m)
\end{equation*}
for all   $m' \in M$.
This action is called the $M$-\emph{shift}, or simply the \emph{shift},  on $X$.
 \end{example}  

 Let $X$ be an $M$-set.
\par
One says that a subset $Y \subset X$ is   $M$-\emph{invariant},
or simply \emph{invariant},  
if  $m y \in Y$ for all $m \in M$ and $y \in Y$. 
If $Y \subset X$ is invariant, then the action of $M$ on $X$ induces by restriction an action of $M$ on $Y$. 
 \par
 The \emph{orbit} of a point $x \in X$ is the 
invariant  subset $\OO_x \subset X$ defined by
 $$
 \OO_x= \OO_x^X :=\{m x: m \in M\}.
 $$
 The \emph{stabilizer relation} of the point $x$ is the equivalence relation $\rho_x$ on $M$ defined by 
\begin{equation}
\label{e:def-rho-x}
  \rho_x := \{(m,m') \in M \times M: m x= m' x\} \subset M \times M.
\end{equation}
There is a natural bijection from $M/\rho_x$ onto $ {\mathcal O}_x$
  given by   $[m] \mapsto m x$.
   \par
The set 
$\gamma := \bigcap_{x \in X}\rho_x \subset M \times M$ is a congruence relation on $M$.
It is the kernel congruence relation  of the monoid morphism from $M$ into  $\Map(X)$ associated with the action.  
One says that the action of $M$ on $X$  is \emph{faithful} if $\gamma = \Delta_M$, i.e., if the monoid morphism $M \to \Map(X)$ associated with the action  is injective.

\begin{example}
\label{ex:shift-is-faithful}
Suppose that   $A$ is a set with more than one element.
Then the shift action of $M$ on $A^M$ is faithful.
Indeed, let $m_1$ and $m_2$ be two distinct elements  in $M$.
Choose   $a,b \in A$ with $a \not= b$ and
consider the configuration $x \in A^M$ defined by $x(m_1) = a$ and $x(m) = b$ for all $m \in M \setminus \{m_1\}$. 
We then have   $(m_1x)(1_M) =
x(m_1) = a$ and
$   (m_2x)(1_M) = x(m_2) = b$,
 so that $(m_1x)(1_M) \neq (m_2x)(1_M)$ and hence $m_1 x \neq m_2 x$. 
 \end{example}

Let   $X$ a $M$-set.
  We say that a point $x \in X$ is \emph{periodic} if its orbit
${\mathcal O}_x$ is finite.
This amounts to saying that the stabilizer  relation $\rho_x$ defined by \eqref{e:def-rho-x} is of finite index.
We  denote by $\Per(X)$ the set consisting of all periodic points in $X$.
\par 
 Given a congruence relation $\gamma \subset M \times M$ on $M$, we set
\[
\Inv(\gamma) = \Inv_X(\gamma):= \{x \in X:  m x =  m' x \mbox{ for all } (m,m') \in \gamma\} =
\{x \in X: \gamma \subset \rho_x\}.
\]

\begin{example}
\label{x:inv-gamma-shift}
Let $A$ be a set and let $\gamma$ be a congruence relation on $M$.
Consider the set $X := A^M$ equipped with the $M$-shift.
 Then $\Inv(\gamma)$ consists of the configurations $x \colon M \to A$
that are constant on each equivalence class of $\gamma$.
Indeed, if $x \in \Inv(\gamma)$ and $(m,m') \in \gamma$, then
$mx = m'x$ and hence
$x(m) = mx(1_M) = m'x(1_M) = x(m')$.
Conversely, if a configuration $x \in X$  is constant on each equivalence class of $\gamma$,
then, for all    $(m,m') \in \gamma$ and $u \in M$,
   we have that 
$(u m ,u m' ) \in \gamma$   and hence
$mx(u) = x(u m) = x(u m') = m'x(u)$ so that  $m x = m' x$,
which shows that  $x \in \Inv(\gamma)$.  
 \end{example}

\begin{proposition}
\label{p:E-periodic-0}
Let $X$ be an $M$-set and suppose that $\gamma \subset M \times M$ is a congruence relation
on $M$.  
Then $\Inv(\gamma)$ is an $M$-invariant subset of $X$.
\end{proposition}

\begin{proof}
Let $x \in \Inv(\gamma)$ and  $m \in M$.
Suppose that $(m',m'') \in \gamma$.
As $\gamma$ is a congruence relation,
  we have that $(m'm,m''m) \in \gamma$ and hence
\[
 m' ( mx) = ( m' m)x =  (m'' m)x =  m'' (  m x).
\]
We deduce that $ mx \in \Inv(\gamma)$. 
This shows that $\Inv(\gamma)$ is $M$-invariant.
\end{proof}

\begin{proposition}
\label{p:per-union-e-gamma}
Let $X$ be an $M$-set.
 Then one has
\begin{equation}
\label{e:per-union-e-gamma}
\Per(X) = \bigcup_{\gamma \in \CC_f(M)} \Inv(\gamma),
\end{equation} 
where $\CC_f(M)$ denotes the set of  finite index congruence relations   on $M$.
\end{proposition}

\begin{proof}
Denote by $E$ the right-hand side of~\eqref{e:per-union-e-gamma}.
Let $\gamma$ be a finite index congruence relation on $M$  and suppose that $x \in \Inv(\gamma)$. 
Then $\gamma \subset \rho_x$ and therefore $\rho_x$ is also of finite index. 
Therefore $x$ is periodic.
This shows that $E \subset \Per(X)$.
\par
Conversely, suppose $x \in \Per(X)$.
Denote by  
$\gamma_x$ the kernel congruence relation of the morphism from $M$ into $\Map(\OO_x)$
associated with the restriction of the action of $M$ on $X$ to the orbit of $x$, i.e., 
 $$
\gamma_x := \{(m,m') \in M \times M : m y = m' y \text{  for all  } y \in \OO_x \}.
 $$
 We then have $x \in \Inv(\gamma_x)$ since $x \in \OO_x$.
On the other hand, the monoid $M/\gamma_x$ is finite, since it is isomorphic to a submonoid of $\Map(\OO_x)$
and  $\OO_x$ is finite.
 We deduce  that $\gamma_x$ is of finite index.
This shows that $\Per(X) \subset E$.
\end{proof}

 Given two $M$-sets $X$ and $Y$, one says that a map $f \colon X \to Y$ is $M$-\emph{equivariant},
 or simply \emph{equivariant},  
if $f(m x) = m f(x)$ for all $m \in M$ and $x \in X$.

\begin{proposition}
\label{p:E-periodic}
Let $X$ and $Y$ be two $M$-sets. 
Suppose that   $f \colon X \to Y$ is an $M$-equivariant map. 
Then the following hold:
\begin{enumerate}[\rm (i)]
\item
$f(\OO_x^X) = \OO_{f(x)}^Y$ for all $x \in X$;
\item
$f(\Per(X)) \subset \Per(Y)$;
\item
if $\gamma$ is a congruence relation on $M$ then $f(\Inv_X(\gamma)) \subset \Inv_Y(\gamma)$.
\end{enumerate}
\begin{comment}
\begin{equation}
\label{e:image-rho-periodiques}
f(\Inv_X(\rho)) \subset \Inv_Y(\rho)
\end{equation}
and
\begin{equation}
\label{e:image-periodiques}
f(\Per(X)) \subset \Per(Y)
\end{equation}
\end{comment}
\end{proposition}

\begin{proof}
Assertion (i) follows from the fact that $f(m x) = mf(x)$ for all $m \in M$.
Assertion (ii) is an immediate consequence of (i) since, by definition, a point is periodic if and only if its orbit is finite.
\par
Let $\gamma$ be a congruence relation on $M$. 
Let $x \in \Inv_X(\gamma)$ and suppose that $(m,m') \in \gamma$.
Then we have that $ mx = m' x$ and hence $f(mx) = f(m' x)$. 
Since $f$ is $M$-equivariant, this gives us $ m f(x) = m'f(x)$. 
We deduce that $f(x) \in \Inv_Y(\gamma)$.
This shows (iii).
\end{proof}

 When $X$ is a topological space, one says that an action   of $M$ on $X$ is \emph{continuous}  if 
the map from $X$ into itself given by $x \mapsto mx$ is continuous for each $m \in M$.

\begin{example}
\label{ex:prodiscrete-top}
Given two sets $A$ and $E$, the \emph{prodiscrete topology} on $A^E = \prod_{e \in E} A$ is the product topology obtained by taking the discrete topology on each factor $A$ 
of $A^E$. If $x \colon E \to A$ is an element of $A^E$, then a base of neighborhoods of $x$ for the prodiscrete topology is provided by the sets
$$
 V(x,F) := \{y \in A^E : x\vert_F = y\vert_F\},
$$
where $F$ runs over all finite subsets of $E$ (we use the notation  $x\vert_F$ to denote the restriction of $x$ to $F$).
The prodiscrete topology  is Hausdorff and totally disconnected.
In the particular case when $A$ is finite,  $A^E$ is compact for the prodiscrete topology  by the Tychonoff product theorem. 
\par
If $M$ is a monoid and $A$ is a set,
then the shift action of $M$ on $A^M$ is  continuous for the prodiscrete topology.
Indeed, suppose that  $x \in A^M$ and $m \in M$.
Then, for  any  finite subset $F$ of $M$, 
   the set $F m$ is also finite. 
Moreover, if $y \in A^M$ coincides with $x$ on $F m$ then $m y$ coincides with $m x$ on $F$.   
\end{example}

\begin{proposition}
\label{p:inv-gamma-closed}
Suppose that  $X$ is a Hausdorff topological space equipped with a continuous action of the monoid $M$
and let $\gamma$ be a congruence relation on $M$.
Then $\Inv(\gamma)$ is a closed subset of $X$.
\end{proposition}

\begin{proof}
Since $X$ is Hausdorff, the continuity of the action implies that
if we fix $m,m' \in M$, then the set of $x \in X$ such that $m x = m' x$ is closed in $X$.
This shows that  the set $\Inv(\gamma)$ is the intersection of a family of closed subsets of $X$ and hence closed in $X$.
\end{proof}

  \subsection{Local and residual properties of monoids}
Let $\PP$ be a property of monoids, that is,
a property that is either satisfied or not for each monoid $M$ and depending only on the isomorphism type of $M$.
 \par
One says that a monoid $M$ is \emph{locally} $\PP$ if every finitely generated submonoid of $M$ 
satisfies  $\PP$. 
 \par
One says that a monoid $M$ is \emph{residually} $\PP$ if, given any pair of distinct elements $m,m' \in M$, there exists a monoid $N$ satisfying $\PP$ and a monoid morphism
$\varphi \colon M \to N$ such that $\varphi(m) \not= \varphi(m')$.
One says that a monoid $M$
is \emph{fully residually} $\PP$ if
for every finite subset $K \subset M$
there exists a monoid $N$ satisfying $\PP$ and a monoid morphism
$\varphi \colon M \to N$ such that the restriction $\varphi\vert_K$ of $\varphi$ to $K$
is injective. 
Clearly every fully residually $\PP$ monoid is   residually $\PP$.
We have the following partial converse.
 \begin{proposition} 
\label{p:res-p}
Let   $\PP$ be a property of monoids. 
Suppose that $\PP$ is closed under finite direct products
(i.e., if $M_1$ and $M_2$ are monoids satisfying $\PP$ then the monoid $M_1 \times M_2$ also satisfies $\PP$).  
 Then every residually $\PP$ monoid is fully residually $\PP$.
\end{proposition}

\begin{proof}
Let $M$ be a residually $\PP$ monoid and let $K$ be a finite subset of $M$.
Consider the set $D$ defined by 
$$
D := \{\{m,m'\} : m,m' \in K \text{ and } m \not= m'\}.
$$
As $M$ is residually $\PP$, for each $d= \{m,m'\} \in D$, there exist a monoid $N_d$ satisfying $\PP$ with a monoid morphism $\psi_d \colon M \to N_d$ such that $\psi_d(m) \not= \psi_d(m')$.
By our hypothesis on $\PP$, 
the product monoid $P := \prod_{d \in D} N_d$ satisfies $\PP$.
Since the product monoid morphism $\varphi := \Pi_{d \in D} \psi_d \colon M \to P$ is injective on $K$,
this shows that $M$ is fully residually $\PP$.
\end{proof}

\begin{corollary}
\label{c:res-finite-fully}
Every residually finite monoid is fully residually finite.
\qed
\end{corollary}

\subsection{Symbolic  characterization of residually finite monoids}
The following statement is well known, at least in the case of groups  (see e.g.~\cite[Theorem~2.7.1]{livre}).

\begin{proposition}
\label{p:res-finite-dynam}
Let $M$ be a monoid. 
Then the  following conditions are equivalent:
\begin{enumerate}[{\rm (a)}]
\item %(a) 
the monoid $M$ is residually finite;
\item %(b) 
for every set $A$, the  set   of  periodic configurations of $A^M$   is
dense in $A^M$ for the prodiscrete topology;
\item %(c) 
there exists a set $A$ having more than one element such that 
the set of  periodic configurations of $A^M$  is
dense in $A^M$ for the prodiscrete topology;
\item %(d) 
there exists a Hausdorff topological space $X$ equipped with a continuous and faithful
action of   the monoid $M$ such that  the set of  periodic points of $X$ is dense in $X$;
\item %(e) 
the intersection of all finite index congruence relations  on $M$ is reduced to the diagonal $\Delta_M
\subset M \times M$.
\end{enumerate}
\end{proposition}

\begin{proof}
Suppose that $M$ is residually finite. 
Let $A$ be a set.  
To prove that periodic configurations are dense in $A^M$, it suffices to  show that,
for every  $x \in A^M$ and any finite subset  $F \subset M$,
 there exists a periodic configuration  $y \in  A^M$ which coincides with $x$ on $F$.
To see this, first observe that it follows from Corollary~\ref{c:res-finite-fully} that we can find a finite monoid
$N$ and a  monoid morphism $\varphi \colon M \to N$ whose restriction to  $F$ is injective.
After replacing $M$ by $\varphi(M)$, we can assume that $\varphi$ is surjective.
Let $\gamma$ denote the kernel congruence relation of $\varphi$ and identify $N$ with $M/\gamma$.
For all $z \in A^N$, 
the configuration   $z \circ \varphi$ is constant on each $\gamma$-class and hence
$ z \circ \varphi\in \Inv(\gamma)$ (see Example~\ref{x:inv-gamma-shift}).
As  $\gamma$ is of finite index since $N$ is finite, 
it follows from Proposition~\ref{p:per-union-e-gamma} that every configuration in $\Inv(\gamma)$ is periodic.
Thus, $z \circ \varphi$ is periodic for all $z \in A^N$. 
On the other hand,
 since $\varphi\vert_F$
is injective, we can find a configuration $z \in A^N$ such that
  $y := z \circ \varphi$ coincides with $x$   on $F$.
Consequently, the configuration $y$ has the required properties.
  This shows that (a) implies (b).
\par
Condition (b) trivially implies (c).
\par
  As $A^M$ is Hausdorff and the  shift action of $M$ on $A^M$ is both continuous
  and faithful if $A$ has more than one element
  (cf. Example~\ref{ex:prodiscrete-top} and Example~\ref{ex:prodiscrete-top}),
condition (c) implies  (d).
\par
 Suppose (d).
It follows from the result of Proposition~\ref{p:per-union-e-gamma} that, for every $x \in \Per(X)$,  there exists a finite index congruence relation $\gamma$ on $M$ such that
 $x \in \Inv(\gamma)$. 
 Thus, if a pair $(m,m')$ belongs to the intersection of all finite index congruence relations on $M$,
 then $m x = m' x$ for all $x \in \Per(X)$.
 As $\Per(X)$ is dense in $X$ and $X$ is Hausdorff we deduce that $m x = m' x$ for all $x \in X$ by continuity of the action.
It follows that $m = m'$ since the action is faithful.
We deduce that the intersection of all finite index congruence relations on $M$ is reduced to $\Delta_M$.
 This shows that (d) implies (e).
\par
Finally, suppose (e). 
Let $m_1$ and $m_2$ be two distinct elements in $M$. 
By (e),  we can find  a finite index congruence relation 
$\gamma \subset M \times M$ such that $(m_1,m_2) \notin \gamma$. 
Then the quotient monoid morphism $\pi_\gamma \colon M \to M/\gamma$ satisfies $\pi_\gamma(m_1) \neq \pi_\gamma(m_2)$. 
Consequently, $M$ is residually finite.
This shows that (e) implies (a).
\end{proof}

\begin{remark}
Let $\Sigma$ be a finite set and let $M:=\Sigma^*$ denote the free monoid on $\Sigma$.
Then $M$ is residually finite. Indeed,   every submonoid of a residually finite monoid is itself residually finite. Moreover every free monoid is a submonoid of a free group and free groups are known to be residually finite (see e.g. \cite[Theorem~2.3.1]{livre}).
In \cite{tree-shifts}, it is shown, by purely combinatorial arguments, that if $A$ is a finite alphabet set and $X \subset A^M$ is a \emph{sofic tree shift} (e.g. $X = A^M$), then the periodic configurations in $X$ are dense in $X$. 
\end{remark}

% SECTION 3
\section{Background material on uniform spaces}
\label{s:background-unif}

This section contains the material on uniform spaces that is needed in the sequel.
The theory of uniform spaces was developed by Weil in~\cite{weil}.
For a  detailed treatment,  
the reader is referred to \cite{weil}, \cite[Ch. 2]{bourbaki}, \cite[Ch. 6]{kelley},  \cite{james},
and \cite[Appendix B]{livre}.

\subsection{Uniform spaces}
Let $X$ be a set. 
\par 
We denote by $\PP(X) := \{A: A \subset X\}$ the set of all subsets of $X$ and by $\Delta_X := \{ (x,x) : x \in X \}$ the diagonal in $X \times X$.
\par
The \emph{inverse} $\overset{-1}{U}$ of a subset $U \subset X \times X$ is the subset of $X \times X$ defined by
$\overset{-1}{U} := \{ (x,y) : (y,x) \in U \}$.
One says that $U$ is \emph{symmetric} if $\overset{-1}{U} = U$.
Note that $U \cap \overset{-1}{U}$ is symmetric for any $U \subset X \times X$.
\par  
We define the \emph{composite} $U \circ V$ of two subsets $U$ and $V$ of $X \times X$  by
$$
U \circ V := \{ (x,y): \text{ there exists  } z \in X \text{  such that  } (x,z) \in U 
\text{  and  } (z,y) \in V \}  \subset X \times X.
$$

\begin{definition}
Let $X$ be a set. A \emph{uniform structure}\index{uniform ! --- structure} on $X$ is a non--empty set $\UU$ of subsets of $X \times X$ satisfying the following conditions:
\begin{enumerate}[(UN-1)]
\item if $U \in \UU$, then $\Delta_X \subset U$;
\item if $U \in \UU$ and $U \subset V \subset X \times X$, then $V \in \UU$;
\item if $U \in \UU$ and $V \in \UU$, then $U \cap V \in \UU$;
\item if $U \in \UU$, then $\overset{-1}{U} \in \UU$;
\item if $U \in \UU$, then there exists $V \in \UU$ such that 
$V \circ V \subset U$.
\end{enumerate}
The elements of $\UU$ are then called the \emph{entourages} of the uniform structure and the set $X$ is called a \emph{uniform space}.
\end{definition}

Note that conditions (UN-3), (UN-4), and (UN-5) imply that, for any entourage $U$ there exists a symmetric entourage $V$ such that $V \circ V \subset U$.
\par
Let $U \subset X \times X$. Given a point $x \in X$, we define the subset $U[x] \subset X$ by
$U[x] := \{y \in X : (x,y) \in U \}$.
\par
If $X$ is a uniform space, there is an induced topology on $X$ characterized by the fact that the neighborhoods of an arbitrary point $x \in X$ consist of the sets $U[x]$, where $U$ runs over all entourages of $X$.
This topology is Hausdorff if and only if the intersection of all the entourages of $X$ is reduced to the diagonal $\Delta_X $.
\par
Let $\UU$ be a uniform structure on a set $X$.
\par
If $Y$ is a subset of $X$, then $\UU_Y = \{V \cap (Y \times Y) : V \in \UU\}$ is a uniform structure on $Y$, which is said to be 
\emph{induced} by $\UU$. The topology on $Y$ associated with $\UU_Y$ is the topology induced by the topology on $X$ associated with $\UU$.
\par
A subset $\BB \subset \UU$ is called a \emph{base} of $\UU$ if for each $W \in \UU$ there exists
$V \in \BB$ such that $V \subset W$. Conversely (cf. \cite[Proposition B.1.6]{livre}), if $X$ is a set and $\BB$ is a nonempty set of subsets of $X  \times X$ satisfying the following properties:
\begin{enumerate}[\rm (BU-1)]
\item 
if $V \in \BB$, then $\Delta_X \subset V$;
\item 
if $V \in \BB$ and $W \in \BB$, then there exists $U \in \BB$ 
such that $U \subset V \cap W$;
\item 
if $V \in \BB$, then there exists $W \in \BB$ such that $W 
\subset \overset{-1}{V}$;
\item 
if $V \in \BB$, then there exists $W \in \BB$ such that $W 
\circ W \subset V$,
\end{enumerate}
then $\BB$ is a base for a unique uniform structure $\UU$ on $X$.

\subsection{Uniformly continuous maps}
Let $X$ and $Y$ be uniform spaces. A map $f \colon X \to Y$ is called
\emph{uniformly continuous} if for each entourage $W$
of $Y$, there exists an entourage $V$ of $X$ such that $(f \times f)(V) \subset W$. Here $f \times f$ denotes the map from $X \times X$ into $Y \times Y$ defined by
$(f \times f)(x_1,x_2) := (f(x_1),f(x_2))$ for all $(x_1,x_2) \in X \times X$.
Every uniformly continuous map $f \colon X \to Y$ is continuous (with respect to the topologies
on $X$ and $Y$ induced by the uniform structures).
Conversely, a continuous map between uniform spaces may fail to be uniformly continuous.
However, we have the following result.

\begin{proposition}
\label{p:cont-comp-implies-unifcont}
Let $X$ and $Y$ be uniform spaces and suppose that $X$ is compact.
Then every continuous map $f \colon X \to Y$ is uniformly continuous.
 \end{proposition}

\begin{proof}
See e.g. \cite[Theorem B.2.3]{livre}.
\end{proof}

 Let $X$ and $Y$ be uniform spaces.
One says that a map $f \colon X \to Y$ is a \emph{uniform isomorphism} if $f$ is bijective and both $f$ and $f^{-1}$ are uniformly continuous.
One says that a map $f \colon X \to Y$ is a \emph{uniform embedding}
if $f$ is injective and induces a uniform isomorphism between $X$ and $f(X) \subset Y$. 

\begin{proposition}
\label{p:inj-hom-unif-emb}
Let $X$ and $Y$ be uniform spaces with $X$ compact and $Y$ Hausdorff.
Suppose that $f \colon X \to Y$ is a continuous injective map.
Then $f$ is a uniform embedding.  
\end{proposition}

\begin{proof}
See e.g.~\cite[Theorem B.2.5]{livre}. 
\end{proof}

\subsection{The discrete uniform structure}
Let $X$ be a set.
The \emph{discrete} uniform structure on $X$ is the uniform structure whose  entourages 
are all subsets of  $ X \times  X$ containing the diagonal $\Delta_X$.
It is the finer uniform structure on $X$.
The topology associated with the discrete uniform structure on $X$ is the discrete one.
If $X$ is equipped with its discrete uniform structure and $Y$ is an arbitrary uniform space,
then every map $f \colon X \to Y$ is uniformly continuous.

\subsection{The prodiscrete uniform structure}
 Let $A$ and $E$ be two sets.
For $e \in E$, we denote by $\pi_e \colon A^E \to A$ the projection map $x \mapsto x(e)$.
Let us equip $A$ with its discrete uniform structure.  
  The \emph{prodiscrete} uniform structure on $A^E$ is the coarsest uniform structure on $A^E$ 
  making all projection maps $\pi_e$ ($e \in E$)   uniformly continuous.
 A base of entourages for the prodiscrete uniform structure on $A^E$ is given by the sets 
$W_F \subset A^E \times A^E$, where
\begin{equation}
\label{e:base-pus}
W_F := \{(x,y) \in A^E \times A^E : x\vert_F = y\vert_F\},
\end{equation}
where $F$ runs over all finite subsets of $E$.
\par
If $Y$ is a uniform space and $\BB$ is a base of entourages of $Y$,
a map $f \colon A^E \to Y$ is uniformly continuous if and only if it satisfies the following condition:
for every  $B \in \BB$, there exists a finite subset $F \subset E$ such that if $x,y \in A^E$ coincide on $F$, then $(f(x),f(y)) \in B$.
\par 
The topology associated with the prodiscrete uniform structure on $A^E$ is the prodiscrete topology
(see~Example~\ref{ex:prodiscrete-top}).

\subsection{Uniformly continuous and expansive actions} 
Let $X$ be a uniform space equipped with an action of a monoid   $M$.  
We consider the diagonal action of $M$ on $X \times X$ defined by
$$
m(x,y) = (mx,my)
$$
for all $m \in M$ and $x,y \in X$.

One says that the action of $M$ on $X$ is \emph{uniformly continuous} if the  map   $x \mapsto mx$ is uniformly continuous on $X$ for each $m \in M$.
This is equivalent to saying that, for each $m \in M$,
the set  $m^{-1}(V) := \{(x,y) \in X \times X: m(x,y) \in V\}$ is an entourage of $X$ for every entourage $V$ of $X$.  

The action of $M$ on $X$ is said to be \emph{expansive} if
there exists an entourage $W_0$ of $X$ such that
\begin{equation}
\label{e;expansive}
\bigcap_{m \in M} m^{-1}(W_0) = \Delta_X,
\end{equation}
where $\Delta_X $ is the diagonal in $X \times X$.
Equality \eqref{e;expansive} means that if $x,y \in X$ satisfy $(m x, m y) \in W_0$ for 
all $m \in M$, then $x = y$.
An entourage $W_0$ satisfying \eqref{e;expansive} is then called an \emph{expansivity entourage} 
for the action of $M$ on $X$.

Our basic example of a uniformly continuous and expansive action is provided by the following:

\begin{proposition} 
\label{p;UCE}
Let $M$ be a monoid and let $A$ be a set. 
Then the $M$-shift on $A^M$ is uniformly continuous and expansive with respect to the prodiscrete uniform structure on $A^M$.
\end{proposition}

\begin{proof}
Let $m \in M$.
Let $F \subset M$.
If two configurations $x,y \in A^M$ coincide on $F m$ then $mx$ and $m y$ coincide on $F$.
As $F m$ is finite whenever $F$ is finite,
we deduce  that the map $x \mapsto m x$ is uniformly continuous on $A^M$.
This shows  that  the $M$-shift is uniformly continuous on $A^M$.
\par 
 Expansiveness of the $M$-shift follows  from  the fact that the entourage $W_0$ of $A^M$ 
defined by
$$
W_0 := \{(x,y) \in A^M \times A^M: x(1_M) = y(1_M)\}
$$
is an expansivity entourage. 
Indeed, if $x, y \in A^M$ satisfy $(m x, m y) \in W_0$ for all $m \in M$, then
$$
x(m) = mx(1_M) = m y (1_M) = y(m) 
$$  
for all $m \in M$, and hence $x = y$.
 \end{proof}

 \subsection{The Hausdorff-Bourbaki uniform structure}
Let $X$ be a set.
Suppose  that $R$ is a subset of $X \times X$.
We shall use the following notation. 
 Given a subset $Y \subset X$, we set
\begin{equation}
\label{e:def-im-subset-by-rel}
R[Y] := \bigcup_{y \in Y} R[y] := \{ x \in X : (x,y) \in R \text{  for some  } y \in Y \}.
\end{equation}
 Then we define the subset $\widehat{R} \subset \PP(X) \times \PP(X)$ by
\begin{equation}
\widehat{R} := \left\{(Y,Z) \in \PP(X) \times \PP(X) : Y \subset R[Z] \text{  and  } Z \subset R[Y] \right\}.
\end{equation}

Let $\UU$ be a uniform structure on $X$. One shows (cf. \cite[Proposition B.4.1]{livre}) that the set $\{\widehat{V} : V \in \UU\}$ is a base for a uniform structure on $\PP(X)$, called the
\emph{Hausdorff-Bourbaki uniform structure} on $\PP(X)$. The topology associated with this uniform structure is called the \emph{Hausdorff-Bourbaki topology}.

We shall need the following two results. 

 \begin{proposition}
\label{p:grom-haus-sep-closed}
Let $X$ be a uniform space.  Let $Y$ and $Z$ be closed subsets of $X$.
Suppose that there is a net $(T_i)_{i \in I}$ of subsets of $X$ which converges to both $Y$ and $Z$ with respect to the Hausdorff-Bourbaki topology on $\PP(X)$.
Then one has $Y = Z$.
In particular, the topology
induced by the Hausdorff-Bourbaki topology on the set of closed subsets of $X$ is Hausdorff. 
\end{proposition}

\begin{proof}
See e.g.~\cite[Proposition B.4.3]{livre}.
\end{proof}

\begin{proposition}
\label{p:uc-implies-gh-uc}
Let $X$ and $Y$ be uniform spaces and let $f \colon X \to Y$ be a uniformly continuous map.
Then the map
$f_* \colon \PP(X) \to \PP(Y)$ which sends each subset $A \subset X$ to its image $f(A) \subset Y$ is uniformly continuous with respect to the Hausdorff-Bourbaki uniform structures on $\PP(X)$ 
and $\PP(Y)$.
\end{proposition}

\begin{proof}
See e.g.~\cite[Proposition B.4.6]{livre}.
\end{proof}

% SECTION 4
\section{Cellular automata over monoids}
\label{s:cellular}

\subsection{Equivalent definitions and examples}

We recall that, if we are given  a monoid $M$ and  a set $A$,   the configuration set $A^M$ is equipped with the prodiscrete uniform structure and the shift action of $M$.

 \begin{definition}
\label{def:ca}
Let $M$ be a monoid and $A$ a set.
A map $\tau \colon A^M \to A^M$ is called a \emph{cellular automaton}
over the monoid $M$ and the alphabet $A$ 
if it is  equivariant  and uniformly continuous. 
 \end{definition}

\begin{remark}
The equivariance of $\tau \colon A^M \to A^M$ with respect to the shift action of $M$ means that one has
$\tau(m x) = m \tau(x)$, that is,
$$
\tau(x \circ R_m) = \tau(x) \circ R_m
$$
for all $x \in A^M$ and $m \in M$.
On the other hand, the uniform continuity of $\tau$
with respect tot the prodiscrete uniform structure  means that, for every   $m \in M$, there exists a finite subset $F \subset M$ such that if two configurations $x,y \in A^M$ coincide on $F$ then the configurations $\tau(x)$ and
$\tau(y)$ take the same value at $m$.
\end{remark}

\begin{remark}
When $M$ is a group, the definition of a cellular automaton over $M$ given above is equivalent to the one in \cite{livre},
except that in \cite{livre} we use the action of $M$ on the configuration space $A^M$  given by $(m,x) \mapsto x \circ L_{m^{-1}}$
instead of $(m,x) \mapsto x \circ R_m$.
So, strictly speaking,
a cellular automaton over a group $M$ according to the definition above is a cellular automaton over the opposite group of $M$ in the sense of the definition in \cite{livre}.  
\end{remark}

\begin{proposition}
\label{p:carc-ca-finite-alphabet}
Let $M$ be a monoid and $A$ a finite set.
Then a map $\tau \colon A^M \to A^M$  
 is a cellular automaton
if and only if it is equivariant  and continuous.   
\end{proposition}

\begin{proof}
The necessity follows  (without any hypothesis on $A$) 
from the fact that  every uniformly continuous map is continuous.
The converse implication  is a consequence of Proposition~\ref{p:cont-comp-implies-unifcont} since $A^M$ is compact by Tychonoff's theorem whenever $A$ is finite.
\end{proof}

 When the monoid $M$ is finite, the result of Proposition~\ref{p:carc-ca-finite-alphabet} remains trivially 
valid even  for an infinite alphabet $A$ because, in that case, the prodiscrete uniform structure on $A^M$ is just the discrete uniform structure, so that every map from $A^M$ into itself is uniformly continuous.
 However, when $M$ is infinite, the following example shows that Proposition~\ref{p:carc-ca-finite-alphabet} becomes false if we remove the hypothesis that $A$ is infinite
(cf.~\cite[Section~4]{ccs-curtis-hedlund} and \cite[Example~1.8.2]{livre}).

\begin{example}
Suppose that $M$ is an infinite monoid and take $A := M$.
Consider the map $f \colon A^M \to A^M$ defined by
$$
f(x)(m) := x(R_m(x(m)))
$$ 
for all $x \in A^M$ and $m \in M$.
Then, for all $m, m' \in M$, we have that
\begin{align*}
f(m x)(m') 
&= m x(R_{m'} ((m x))(m')) \\ 
&= m x(R_{m'} ( x(m' m) )) \\
&=  x(   x(m' m) m' m) \\ 
& = x( R_{m' m}(  x(m' m) )) \\
& = f(x)(m' m) \\ 
& = (m f(x))(m')
\end{align*}
and hence $f(m x) = m f(x)$.
This shows that $f$ is equivariant.
\par
Moreover, $f$ is continuous.
In other words,   for all $x \in A^M$ and $m \in M$, there exists a finite subset $F \subset M$ such that
if a configuration $y \in A^M$ coincides with $x$ on $F$, then $\tau(x)$ and $\tau(y)$ take the same value at $m$.
Indeed, we can take for example 
$$
F := \{x(m)\} \cup \{R_m(x(m))\}.
$$
\par
However, $f$ is not a cellular automaton.
To see this, suppose that we are given a finite subset  $F \subset M$.
Choose an element $m_0 \in M \setminus (F \cup \{1_M\})$ (this is possible because $M$ is assumed to be infinite)
and consider the configurations $x,y \in A^M$ respectively defined by
 $$
x(m) := 
\begin{cases}
1_M & \mbox{ if \ } m = m_0\\
 m_0 & \mbox{ if } m \in M \setminus \{m_0\},
\end{cases}
$$
and
$$
y(m) := m_0 \quad \text{for all } m \in M.
$$
Note that $x$ and $y$ coincide on $M \setminus \{m_0\}$ and hence on $F$.
However, we have that
\[
f(x)(1_M) = x(x(1_M)) = x(m_0) = 1_M
\]
while
\[
f(y)(1_M) = y(y(1_M)) = y(m_0) = m_0,
\]
so that $f(x)(1_M) \neq f(y)(1_M)$. 
It follows that there is
no finite subset $F \subset M$ with the property  that
if two configurations in $A^M$ coincide on $F$,
then their images by $f$  take the same value at $1_M$. 
 This shows that $f$ is not uniformly continuous and hence that 
$f$ is not a cellular automaton.  
 \end{example}

The following algebraic characterization of cellular automata on monoids
  (with finite or infinite alphabet)
 extends the classical Curtis-Hedlund-Lyndon theorem \cite{hedlund} as well as its generalization in~\cite{ccs-curtis-hedlund}
(see also~\cite[Theorem~1.9.1]{livre}).

 \begin{theorem}
\label{t:ca-iff-equiv-unif-cont}
Let $M$ be a monoid and  $A$ a set. 
Let $\tau \colon A^M \to A^M$ be a map.
Then the following conditions are equivalent:
\begin{enumerate}[\rm (a)]
\item
$\tau$ is a cellular automaton;
\item
there exists a finite subset $S \subset M$ and a map $\mu \colon A^S \to A$ 
such that
\begin{equation}
\label{e:cell-aut}
\tau(x)(m) = \mu\left((mx)\vert_S\right)
\end{equation}
for all $x \in A^M$ and $m \in M$.
(Recall that $(mx)\vert_S \in A^S$ denotes the restriction of the configuration $m x \in A^M$ to 
  $S \subset M$.)
  \end{enumerate}
\end{theorem}

If $S \subset M$ and $\mu \colon A^S \to A$ are as in the statement of  the above theorem, one  says that $S $ is a \emph{memory set} for $\tau$ and that $\mu$ is the associated \emph{local defining map}.
Note that $\mu$ is entirely determined by $\tau$ and $S$, 
since the restriction map $A^M \to A^S$ is surjective and $\mu(x\vert_S) = \tau(x)(1_M)$ for all $x \in A^M$ by \eqref{e:cell-aut}.
 
\begin{proof}[Proof of Theorem \ref{t:ca-iff-equiv-unif-cont}]
Suppose first that $\tau \colon A^M \to A^M$ is a cellular automaton.
 Since $\tau$ is uniformly continuous for the prodiscrete uniform structure on $A^M$, there exists a finite subset
$S \subset M$ such that,
if two configurations in $A^M$ coincide on $S$, then their images by $\tau$ take the same value at $1_M$.
In other words, there is a map $\mu \colon A^S \to A$ such that
\[
\tau(x)(1_M) = \mu(x\vert_S)
\]
for all $x \in A^M$.
By using the $M$-equivariance of $\tau$, we get
\[
\begin{split}
\tau(x)(m)  & = \left(m\tau(x)\right)(1_M)\\
& = \tau(mx)(1_M)\\
& = \mu\left((mx)\vert_{S}\right)
\end{split}
\]
for all $m \in M$. 
This shows that (a) implies (b).
\par
Conversely, suppose that $S \subset M$ is a finite subset and $\mu \colon A^S \to A$ is a map   satisfying~\eqref{e:cell-aut}. 
Let $m,m' \in M$ and $x \in A^M$. We have that
\[
\begin{split}
\tau(mx)(m') & = \mu\left((m'mx)\vert_S\right)\\ 
& = \tau(x)(m'm)\\
&  = \left(m \tau(x)\right)(m').
\end{split}
\]
It follows that $\tau(mx) = m\tau(x)$ for all $m \in M$ and $x \in A^M$.
This shows that $\tau$ is $M$-equivariant.
On the other hand,
we deduce  from  \eqref{e:cell-aut} that if two configurations $x,y \in A^M$ coincide on 
$Sm$ for some $m \in M$, then $\tau(x)$ and $ \tau(y)$ take the same value at $m$.
As the set $Sm$ is finite for each $m \in M$,
 we conclude that  $\tau$ is uniformly continuous with respect to the prodiscrete uniform structure on $A^M$.
This shows that  $\tau$ is a cellular automaton.
\end{proof}

Let us give  some  examples of cellular automata that can be defined over an arbitrary monoid $M$.

 \begin{examples}
\label{ex}
(a) Let $A$ be a set. Then the  identity map $\tau = \Id_{A^M} \colon A^M \to A^M$ is a cellular automaton. 
 We can take $S := \{1_M\}$
as a memory set for $\tau$ and the identity map $\mu := \Id_A \colon A^S\equiv A \to A$ as the associated
local defining map.
\par
(b) Let $A$ be a set. Fix a map  $f \colon A \to A$.    
Then the   map $\tau \colon A^M \to A^M$, defined by   $\tau(x) := f \circ x$ 
for all $x \in A^M$,   is a cellular automaton.
Here we can take again $S := \{1_M\}$ as a memory set for $\tau$
and $\mu := f \colon A^S\equiv A \to A$ as the associated local defining map.
Observe that if $f$ is the identity map on $A$ then $\tau$ is the identity map on $A^M$.
\par
(c)   Let $A$ be a set. Fix an element $m \in M$.
Then the map $\tau  \colon A^M \to A^M$, defined by
  $\tau(x) := x \circ L_{m}$ for all
$x \in A^M$,     is a cellular automaton.
 Here we can take $S := \{m\}$ as a memory set for $\tau$ and the identity map $\mu := \Id_A \colon A^S\equiv A \to A$ as the associated local defining map.
Note that if $m = 1_M$ then   $\tau $ is the identity map on $A^M$.
\par
(d) Let $A = \{0,1\}$ and $S_0 \subset M$ be a finite subset.
The map  $\tau \colon A^M \to A^M$, defined by   
\[
\tau(x)(m) := \begin{cases}
1 & \mbox{ if } \sum_{s\in S_0} x(sm) > |S_0|/2\\
x(m) & \mbox{ if } \sum_{s\in S_0} x(sm) = |S_0|/2\\
0 & \mbox{ otherwise}
\end{cases}\]
for all $x \in A^M$ and $m \in M$, is a cellular automaton.
It is called the \emph{majority action} cellular automaton on $M$ relative to $S_0$.
A memory set for $\tau$ is  $S:=S_0 \cup\{1_M\}$ and the associated local defining map $\mu \colon A^S \to A$ is given by
\[
\mu(y) := \begin{cases}
1 & \mbox{ if } \sum_{s\in S_0} y(s) > |S_0|/2\\
y(1_M) & \mbox{ if } \sum_{s\in S_0} y(s) = |S_0|/2\\
0 & \mbox{ otherwise}
\end{cases}\]
for all $y \in A^S$.  
  \par
(e) Let $A$ be a field and $S_0 \subset M$ a  finite subset whose cardinality
$|S_0|$ is not an integral  multiple of the characteristic of $A$.
Then the  map  $\tau \colon A^M \to A^M$, defined by 
\[
\tau(x)(m) := x(m) - |S_0|^{-1} \sum_{s \in S_0} x(sm)
\]
for all $x \in A^M$ and $m \in M$, is a cellular automaton. It is called the
\emph{combinatorial laplacian}   on $M$ with coefficients in $A$  relative to $S_0$.
We can take $S:= S_0 \cup \{1_M\}$ as a memory set for $\tau$ and $\mu \colon A^S \to A$ defined by
\[
\mu(y) := y(1_M) - |S_0|^{-1} \sum_{s \in S_0} y(s)
\]
for all $y \in A^S$,   as the associated local defining map.
 \end{examples}

 \subsection{Composition of cellular automata}
Given a monoid $M$ and a set $A$, we denote by $\CA(M;A)$ the set consisting of all cellular automata 
$\tau \colon A^M \to A^M$.
It turns out that $\CA(M;A)$ is a monoid for the composition of maps. More precisely, we have the following result.\begin{proposition}
Let $M$ be a monoid and $A$ a set.
Then the set $\CA(M;A)$ is a submonoid of the symmetric monoid $\Map(A^M)$. 
 \end{proposition}

\begin{proof}
In Example~\ref{ex}.(a), we have observed that the
identity map $\Id_{A^M} \colon A^M \to A^M$ is a cellular automaton. 
On the other hand,
if $\tau_1, \tau_2 \colon A^M \to A^M$ are two cellular automata
then the composition map $\tau_1 \circ \tau_2 \colon A^M \to A^M$ is also a cellular automaton,
since the composition of two equivariant (resp.~uniformly continuous) maps is clearly equivariant (resp. uniformly continuous).
\end{proof}

\begin{remark}
\label{r:composition}
Given two cellular automata $\tau_1, \tau_2 \colon A^M \to A^M$,
we can express a memory set and the associated local defining map for
$\tau_1 \circ \tau_2$ in terms of memory sets and their associated local defining maps for
$\tau_1$ and $\tau_2$ as follows.

Let $S_i \subset M$ and $\mu_i \colon A^{S_i} \to A$ be a memory set and associated local
defining map for $\tau_i$, $i=1,2$.

For $y \in A^{S_2S_1}$ and $s_1 \in  S_1$, define $y_{s_1} \in A^{S_2}$ by setting $y_{s_1}(s_2) = y(s_2s_1)$ for all $s_2 \in S_2$. Also, denote by $\overline{y} \in A^{S_1}$ the map defined by $\overline{y}(s_1) = \mu_2(y_{s_1})$ for all $s_1 \in S_1$. We finally define the map $\mu \colon A^{S_2S_1} \to A$ by setting
\begin{equation}
\label{e;mu-phi-mu'}
\mu(y) = \mu_1(\overline{y})
\end{equation}
for all $y \in A^{S_2S_1}$.
Let $x \in A^M$, $m \in M$, $s_1 \in S_1$, and $s_2 \in S_2$. We then have
\[
\begin{split}
(s_1mx)\vert_{S_2}(s_2) & = s_1mx(s_2)\\
& = mx(s_2s_1) \\
& = (mx)\vert_{S_2S_1}(s_2s_1)\\
&  = \left((mx)\vert_{S_2S_1}\right)_{s_1}(s_2).
\end{split}
\]
It follows that
\begin{equation*}
(s_1mx)\vert_{S_2} = \left((mx)\vert_{S_2S_1}\right)_{s_1}
\end{equation*}
and therefore
\begin{equation*}
\tau_2(mx)(s_1) = \mu_2\left((s_1mx)\vert_{S_2}\right) = \mu_2\left(\left((mx)\vert_{S_2S_1}\right)_{s_1}\right) = \overline{(mx)\vert_{S_2S_1}}(s_1).
\end{equation*}
As a consequence,
\begin{equation}
\label{e:mu-mu1-mu2}
\tau_2(mx)\vert_{S_1} = \overline{(mx)\vert_{S_2S_1}}.
\end{equation}
Finally, one has
\begin{equation}\label{e;mu-tau-circ-tau}
\begin{split}
(\tau_1 \circ \tau_2)(x)(m) & = \tau_1(\tau_2(x))(m) \\
& = \mu_1\left((m\tau_2(x))\vert_{S_1}\right)\\
\mbox{(by $M$-equivariance of $\tau_2$)} \ \ & = \mu_1(\tau_2(mx) \vert_{S_1})\\
\mbox{(by \eqref{e:mu-mu1-mu2})} \ \ & = \mu_1(\overline{(mx)\vert_{S_2S_1}})\\
\mbox{(by \eqref{e;mu-phi-mu'})} \ \ & = \mu\left((mx)\vert_{S_2S_1}\right).
\end{split}
\end{equation}
This shows that the cellular automaton $\tau_1 \circ \tau_2$ admits $S:= S_2S_1 \subset M$ as memory set and $\mu \colon A^S \to A$ as associated local defining map.
\end{remark}

\subsection{Reversible cellular automata}
Let $M$ be a monoid and $A$ a set.
A cellular automaton $\tau \colon A^M \to A^M$ is called \emph{reversible} if $\tau$ is an invertible element of the monoid $\CA(M;A)$.
Thus, a   cellular automaton $\tau \colon A^M \to A^M$ is  reversible
if and only if $\tau$ is bijective and the inverse map $\tau^{-1} \colon A^M \to A^M$ is also a cellular automaton.

 \begin{proposition}
\label{p:reversible}
Let $M$ be a monoid and $A$ a finite set. 
Then every bijective cellular automaton
$\tau \colon A^M \to A^M$ is reversible.
\end{proposition}

\begin{proof}
Let $\tau \colon A^M \to A^M$ be a bijective cellular automaton.
Since $\tau$ is equivariant,
its inverse map $\tau^{-1}$   is also equivariant.
On the other hand, as $\tau$ is continuous  and $A^M$ is a compact Hausdorff space,
 $\tau^{-1}$ is continuous.  
 Therefore $\tau^{-1}$ is a cellular automaton by Proposition~\ref{p:carc-ca-finite-alphabet}.
\end{proof}

The following example (cf. \cite[Lemma~5.1]{reversible} 
and \cite[Example~1.1.3]{livre}) shows that Proposition~\ref{p:reversible} becomes false if we omit the finiteness hypothesis on the alphabet $A$.

\begin{example}
\label{e:not-invert}
Let $M$ be a monoid containing an element $m_0$ of infinite order (i.e., such that the set of powers of $m_0$ is infinite). 
Let $K$ be a field and
take $A: = K[[t]]$,  the ring of all formal power series  with coefficients in $K$.
 A configuration $x \in A^M$ is a map $x \colon M \to K[[t]]$ such that
\begin{equation}
\label{e:expression-de-x-de-n}
x(m) = \sum_{i \in \N} x_{m,i}t^i,
\end{equation}
where $x_{m,i} \in K$ for all $m \in M$ and $i \in \N$.
Consider the map $\tau \colon A^M \to A^M$ defined by
$$
\tau(x)(m) := x(m) - tx(m_0 m)
$$
for all $x \in A^M$ and $m \in M$.
Observe that $\tau$ is the cellular automaton admitting $S = \{1_M,m_0\}$ as a memory set
and the map $\mu \colon A^S \to A$, defined by $\mu(y) = y(1_M) - ty(m_0)$ for all $y \in A^S$, as the associated local defining map.
\par
Let us show that $\tau$ is bijective.
Consider the map $\sigma \colon A^M \to A^M$ given by
$$
\sigma(x)(m) := x(m) + tx(m_0 m) + t^2x(m_0^2 m) + t^3x(m_0^3 m) + \cdots
$$
for all $x \in A^M$ and $m \in M$. 
 It is straightforward to check that $\sigma \circ \tau = \tau \circ \sigma = \Id_{A^M}$.
Therefore, $\tau$ is bijective with inverse map $\tau^{-1} = \sigma$.
\par
Let us show that the map $\sigma \colon A^M \to A^M$ is not a cellular automaton.
Let $x \in A^M$ be the identically-zero configuration, i.e., the configuration  defined by $x(m) = 0$ for all $m \in M$.
Let $F$ be a finite subset of $M$. 
Choose $k \in \N$ 
such that $m_0^k \notin F$ (this is possible because $m_0$ has infinite order). 
 Consider the configuration $y \in A^M$ defined by $y(m) = 0$ if
 $m \in M \setminus \{m_0^k\}$ and $y(m_0^k) = 1$.   
 Then $y$ and $x$ coincide on $F$.
However, the value at $1_M$ of $\sigma(y)$ is $t^k$
 while the value of $\sigma(x)$ at $1_M$ is $0$. 
As $t^k \not= 0$, we conclude that $\sigma$ is not continuous at $x$ (for the prodiscrete topology).
Since every cellular automaton must be uniformly continuous and hence continuous,
this shows that $\sigma$ is not a cellular automaton.
Thus, $\tau$ is a bijective cellular automaton that is not reversible.
\end{example}

\subsection{Minimal memory}
\label{sec:minimal-memory}

Let $M$ be a monoid and let $A$ be a set.
Let $\tau \colon A^M \to A^M$ be a cellular automaton. Let $S$ be a memory set for $\tau$ and let $\mu \colon A^S \to A$ be the associated defining map. If $S'$ is a finite subset of $M$ such that $S \subset S'$, then $S'$ is also a memory
set for $\tau$ and the local defining map associated with $S'$ is the map
$\mu' \colon A^{S'} \to A$ given by $\mu' = \mu \circ p$, where $p \colon A^{S'} \to A^S$ is the  canonical projection (restriction map). This shows that the memory set of a cellular automaton is not unique in general. However, we shall see that every cellular automaton admits a unique memory set of minimal cardinality. Let us first establish the following auxiliary result.

\begin{lemma}
\label{l:inter-memoire}
Let $M$ be a monoid and $A$ a set.
Let $\tau \colon A^M \to A^M$ be a cellular automaton.
Suppose that  $S_1$ and $S_2$ are memory sets for $\tau$.
Then $S_1 \cap S_2$ is also a memory set for $\tau$.
\end{lemma}

\begin{proof}
Let $x \in A^M$.
We claim that $\tau(x)(1_M)$ depends only on the restriction of $x$ to $S_1 \cap S_2$.
To see this, consider an element $y \in A^M$ such that $x\vert_{S_1 \cap S_2} = y\vert_{S_1 \cap S_2}$. Let us choose an element $z \in A^M$ such that $z\vert_{S_1} = x\vert_{S_1}$
and $z\vert_{S_2} = y\vert_{S_2}$ (we may take for instance the configuration $z \in A^M$ which coincides with $x$ on $S_1$ and with $y$ on $M \setminus S_1$). We have
$\tau(x)(1_M) = \tau(z)(1_M)$ since $x$ and $z$ coincide on
$S_1$, which is a memory set for $\tau$. On the other hand, we have $\tau(y)(1_M) = \tau(z)(1_M)$ since $y$ and $z$ coincide on $S_2$, which is also a memory set for $\tau$. It follows that $\tau(x)(1_M) = \tau(y)(1_M)$. This proves our claim.
\par
Thus there exists a map
$\mu \colon A^{S_1 \cap S_2} \to A$ such that $\tau(x)(1_M) = \mu(x\vert_{S_1 \cap S_2})$ for all 
$x \in A^M$. Since $\tau$ is $M$-equivariant, arguing as in the proof of Theorem \ref{t:ca-iff-equiv-unif-cont}, we deduce that $S_1 \cap S_2$ is a memory set for $\tau$. 
\end{proof}

\begin{proposition}
Let $\tau \colon A^M \to A^M$ be a cellular automaton.
Then there exists a unique memory set $S_0 \subset M$ for $\tau$ of minimal cardinality.
Moreover, if $S$ is a finite subset of $M$, then $S$ is a memory set for $\tau$ if and only
if $S_0 \subset S$.
\end{proposition}
\begin{proof}
Let $S_0$ be a memory set for $\tau$ of minimal cardinality.
As we have seen at the beginning of this section,
every finite subset of $M$ containing $S_0$ is also a memory set for $\tau$.
Conversely, let $S$ be a memory set for $\tau$.
As $S \cap S_0$ is a memory set for $\tau$ by Lemma \ref{l:inter-memoire}, we have 
$\vert S \cap S_0 \vert \geq \vert S_0 \vert$, by minimality of $|S_0|$. 
This implies $S \cap S_0 = S_0$, that is,
$S_0 \subset S$. In particular, $S_0$ is the unique memory set of minimal cardinality.
\qed\end{proof}

The memory set of minimal cardinality of a cellular automaton is called its \emph{minimal}
memory set.

\begin{remark}
A map $F \colon A^M \to A^M$ is constant if there exists a configuration $x_0 \in A^G$ such that
$F(x) = x_0$ for all $x \in A^M$.
By $M$-equivariance, a cellular automaton $\tau \colon A^M \to A^M$ is constant if and only if there exists $a \in A$ such that $\tau(x)(m) = a$ for all $x \in A^M$ and $m \in M$.
Observe that a cellular automaton $\tau \colon A^M \to A^M$ is constant if and only if its minimal memory set is the empty set.
\end{remark}

\subsection{Induction and restriction of cellular automata}
Let $M$ be a monoid and let $A$ be a set. Let also $N$ be a submonoid of $M$. 
\par
We denote by $\CA(M,N;A)$ the set consisting of all cellular automata $\tau \colon A^M \to A^M$
admitting a memory set $S$ such that $S \subset N$.
Note that a cellular automaton $\tau \in \CA(M;A)$ is in $\CA(M,N;a)$ if and only if the minimal memory set of $\tau$ is contained in $N$. 
It follows from Remark~\ref{r:composition}
that if $\tau_1, \tau_2 \colon A^M \to A^M$ are two cellular automata admitting memory sets
$S_1, S_2$ contained in $N$ then their composition $\tau_1 \circ \tau_2$ admits the set
$S:= S_2S_1 \subset N$ as a memory set. Moreover, the identity cellular automaton $\Id_{A^M}$
admits $S:= \{1_M\} \subset N$ as a memory set. As a consequence, $\CA(M,N;A)$ is a submonoid of $\CA(M;A)$. 
\par
Let $\tau \in \CA(M,N;A)$. Let $S$ be a memory set for $\tau$ such that $S \subset N$ and denote by
$\mu \colon A^S \to A$ the associated local defining map. 
Then, by virtue of Theorem \ref{t:ca-iff-equiv-unif-cont}, the map $\tau_N \colon A^N \to A^N$
defined by
\[
\tau_N(x)(n) = \mu((nx)\vert_S)
\]
for all $x \in A^N$ and $n \in N$ is a cellular automaton over the submonoid $N$ with memory set $S$
and local defining map $\mu$. Observe that if $\widetilde{x} \in A^M$ is such that $\widetilde{x}\vert_N = x$, then
\begin{equation}
\label{e:tilde-restriction}
\tau_N(x)(n) = \tau(\widetilde{x})(n)
\end{equation}
for all $n \in N$. This shows, in particular, that $\tau_N$ does not depend on the choice of
the memory set $S \subset N$. We call $\tau_N \in \CA(N;A)$ the \emph{restriction} of the cellular automaton $\tau \in \CA(M,N;A)$ to $N$.

Conversely, let $\sigma \colon A^N \to A^N$ be a cellular automaton over $N$ with memory set $S$
and local defining map $\mu \colon A^S \to A$. Then the map $\sigma^M \colon A^M \to A^M$  
defined by
\[
\sigma^M(\widetilde{x})(m) := \mu((m\widetilde{x})\vert_S)
\]
for all $\widetilde{x} \in A^M$ and $m \in M$ is a cellular automaton over $M$ with memory set $S$
and local defining map $\mu$. If $S_0$ is the minimal memory set of $\sigma$ and $\mu_0 \colon A^{S_0} \to A$ is the associated local defining map, then $\mu = \mu_0 \circ \pi$, where
$\pi \colon A^S \to A^{S_0}$ is the restriction map, and we have
\[
\sigma^M(\widetilde{x})(m) = \mu((m\widetilde{x})\vert_S) = \mu_0 \circ \pi((m\widetilde{x})\vert_S) = \mu_0((m\widetilde{x})\vert_{S_0})
\]
for all $\widetilde{x} \in A^M$ and $m \in M$. This shows, in particular, that $\sigma^M$ does not depend on the choice of the memory set $S \subset N$. 
Note that if $\widetilde{x} \in A^M$ is such that $\widetilde{x}\vert_N = x$, then we have
the analogue of \eqref{e:tilde-restriction}, namely
\begin{equation}
\label{e:tilde-induction}
\sigma^M(\widetilde{x})(n) = \sigma(x)(n)
\end{equation}
for all $n \in N$.

We call $\sigma^M \in \CA(M,N;A)$ the cellular automaton \emph{induced} by $\sigma \in \CA(N;A)$.
\par
The proof of the following result immediately follows from the above definitions.

\begin{proposition}
\label{p:ind-rest-inv}
The map $\tau \mapsto \tau_N$ is a monoid isomorphism from $\CA(M,N;A)$ onto
$\CA(N;A)$ whose inverse is the map $\sigma \mapsto \sigma^M$.
\qed
\end{proposition}

\begin{proposition}
\label{p:ind-res}
Let $M$ be a monoid and let $A$ be a set. Let also $N$ be a submonoid of $M$.
Then the following implications hold:
\begin{enumerate}[{\rm (i)}]
\item 
if $\tau \in \CA(M,N;A)$ is surjective then $\tau_N \in \CA(N;A)$ is surjective; 
\item 
if $\sigma \in \CA(N;A)$ is injective then $\sigma^M \in \CA(M,N;A)$ is injective.
\end{enumerate}
\end{proposition}

\begin{proof}
Suppose first that $\tau \in \CA(M,N;A)$ is surjective and denote by $S \subset N$ and $\mu \colon A^S \to A$ a memory set and associated local defining map. 
Let $y \in A^N$ and let $\widetilde{y} \in A^M$ be a configuration extending $y$. Since $\tau$ is surjective, we can find $\widetilde{x} \in A^M$ such that $\tau(\widetilde{x}) = \widetilde{y}$. Let $x:=\widetilde{x}\vert_N \in A^N$ denote
the restriction of $\widetilde{x}$ to $N$. Then we have
\[
\begin{split}
\tau_N(x)(n) & = \mu((nx)\vert_S)\\
(\mbox{since } S \subset N) \  \ & =  \mu((n\widetilde{x})\vert_S)\\
& = \tau(\widetilde{x})(n)\\
& = \widetilde{y}(n)\\
& = y(n)
\end{split}
\]
for all $n \in N$, that is, $\tau_N(x) = y$. This shows that $\tau_N$ is surjective.
\par
Suppose now that $\sigma \in \CA(N;A)$ is injective and denote by $S \subset N$ and $\mu \colon A^S \to A$ a memory set and associated local defining map. 
Let $x_1, x_2 \in A^M$ and suppose that $\sigma^M(x_1) = \sigma^M(x_2)$. Let us show that $x_1 = x_2$.
Let $m \in M$ and observe that, by $M$-equivariance, we have $\sigma^M(mx_1) = m\sigma^M(x_1) = m\sigma^M(x_2)= \sigma^M(mx_2)$. It follows from \eqref{e:tilde-induction} that
$\sigma((mx_1)\vert_N) = \sigma((mx_2)\vert_N)$. Since $\sigma$ is injective, we deduce that
$(mx_1)\vert_N = (mx_2)\vert_N$, in particular $x_1(m) = (mx_1)\vert_N(1_M) = (mx_2)\vert_N(1_M) =
x_2(m)$. Thus $x_1 = x_2$. This shows that $\sigma^M$ is injective.
\end{proof}

In contrast with the case of cellular automata over groups (see~\cite[Proposition~1.7.4]{livre}),  
the implications in Proposition~\ref{p:ind-res} cannot be reversed as shown by the following example. 

 \begin{example}
 \label{ex:contre-ex-ind-res}
Take $M:= B = \langle p,q: pq=1\rangle$, the bicyclic monoid,
and  $N:= \langle p\rangle$, the infinite cyclic submonoid  generated by $p$.
Let $A$ be a set   and consider the cellular automaton $\tau  \in \CA(M,N;A)$ defined by
$\tau(x) := x \circ L_p$ for all $x \in A^M$ (cf.~Example~\ref{ex}.(c)
and the proof of Theorem~\ref{t:contains-bicyclic-is-non-surj} below).
As $x = \tau(x) \circ L_q$ for all $x \in A^M$,
the cellular automaton $\tau$ is injective.
Moreover, its restriction   $\tau_N \colon A^N \to A^N$ is surjective,
since, if we are given $y \in A^N$, then any  $z \in A^N$  such that $z(p^n) = y(p^{n - 1})$ for all $n \geq 1$ satisfies $\tau(z) = y$.
However, as soon as $A$ has more than one element, $\tau_N$ is not injective and $\tau$ is not surjective.
Indeed,   every configuration in the image of $\tau$ must take the same value at $1_M$ and $q p$.
 On the other hand, any two configurations in $A^N$ that coincide outside of  $1_M$ have the same image by $\tau_N$. 
\end{example}

% section 5
\section{Surjunctive monoids}
\label{sec:surj-monoids}

 \subsection{Definition and first examples}

\begin{definition}
A monoid $M$ is called \emph{surjunctive} if
  every
injective cellular automaton with finite alphabet over $M$ is surjective.
\end{definition}

When $M$ is a group, the above definition is equivalent to the one given by Gottschalk 
\cite{gottschalk} (see also \cite[Chapter 3]{livre}).

\begin{remark}
If a monoid $M$ is surjunctive, then every injective cellular automaton with finite alphabet over $M$ is bijective and hence reversible 
by Proposition~\ref{p:reversible}.
\end{remark}

\begin{proposition}
\label{p:finite-surj}
Every finite monoid is surjunctive.
\end{proposition}

\begin{proof}
Suppose that  $M$ is a finite monoid and $A$ a finite set.
Then $A^M$ is also finite.
As    every injective self-map of a finite set is surjective,
it follows that every injective cellular automaton $\tau \colon A^M \to A^M$ is surjective.
\end{proof}

\begin{corollary}
If $X$ is a finite set then the symmetric monoid $\Map(X)$ is surjunctive.
\qed
\end{corollary}

Note that the monoid $\Map(X)$ is not a group as soon as $X$ contains more than one element.

\subsection{Examples of non-surjunctive monoids}
We recall that it is an open problem to determine whether or not all groups are surjunctive (Gottschalk's conjecture).
In the more general setting of monoids, 
   the answer to this question is negative  
  (cf.~\cite[Example 8.2]{semieden}):

 \begin{theorem}
 \label{t:contains-bicyclic-is-non-surj}
Every monoid containing a submonoid isomorphic to the bicyclic monoid is non-surjunctive.
In particular,
the bicyclic monoid is non-surjunctive.
 \end{theorem}

 \begin{proof}
Let $M$ be a monoid containing a submonoid isomorphic to the bicyclic monoid $B$.
Then we can find elements $p,q \in M$ such that $pq = 1_M \not= qp $. 
\par
Let $A$ be any finite set and consider the cellular automaton 
$\tau   \colon A^M \to A^M$
defined by   $\tau(x) := x \circ L_p$ for all $x \in A^M$  (cf. Example~\ref{ex}.(c)).
Observe first that $\tau$ is injective since
$$
\tau(x) \circ L_q = x \circ L_p \circ L_q = x \circ L_{1_M} = x
$$
for all $x \in A^M$.
 \par
On the other hand, we have that
\[
\tau(x)(1_M) = x \circ L_p (1_M) = x(p) = x(pqp) = x \circ L_p (q p) = \tau(x)(q p)
\] 
for all $x \in A^M$. 
Therefore, any configuration lying in the image of $\tau$ must take the same value at $1_M$ and $qp$.
As $1_M \not= q p$, it follows that $\tau$ is not surjective as soon as $A$ has more than one element.
\par
This shows that $M$ is not surjunctive.
\end{proof}

 Recall from Proposition~\ref{p:charact-bicyclic-supmonoid} that a monoid contains a submonoid isomorphic to the bicyclic monoid if and only if it admits an element that is invertible from one side but not from the other.
Thus, the following statement is   a reformulation of Theorem~\ref{t:contains-bicyclic-is-non-surj}.

\begin{corollary}
Every monoid containing an element that is invertible from one side but not from the other is non-surjunctive.
\qed
\end{corollary}

\begin{corollary}
\label{c:map-X-not-surj}
Let $X$ be an infinite set. Then the symmetric monoid $\Map(X)$ is non-surjunctive.
\end{corollary}
\begin{proof}
As $X$ contains a countably-infinite set, 
the monoid  $\Map(X)$ contains a submonoid isomorphic to the monoid $\Map(\N)$ and 
hence a submonoid isomorphic to the bicyclic monoid.
\end{proof}

If $V$ is a vector space over a field $K$, then the set
$\End_K(V)$ consisting of all $K$-linear maps $u \colon V \to V$ is a submonoid of $\Map(V)$ and hence a monoid for the composition of maps.

\begin{corollary}
Let $K$ be a field and let $V$ be an infinite-dimensional vector space over $K$.
Then the monoid $\End_K(V)$ is non-surjunctive.
\end{corollary}

\begin{proof}
 Let $(e_x)_{x \in X}$ be a basis for $V$.
Then the symmetric monoid $\Map(X)$ embeds into $\End_K(V)$ via the map that sends each $f \in \Map(X)$ to the unique $K$-linear map $u \colon V \to V$  such that $u(e_x) = e_{f(x)}$ for all $x \in X$.
As $X$ is infinite, we conclude that $\End_K(V)$ contains a submonoid isomorphic to $\Map(\N)$ and hence a submonoid isomorphic to the bicyclic monoid.
\end{proof}

\begin{remark}
It would be interesting to give an example of a cancellative non-surjunctive monoid. Note that the
bicyclic monoid is neither left-cancellative  nor right-cancellative.
 \end{remark}

\subsection{Submonoids of surjunctive monoids}

\begin{theorem}
\label{t:sub-surjunctive}
Every submonoid of a surjunctive monoid is itself surjunctive.
\end{theorem}

\begin{proof}
Let $M$ be a surjunctive monoid.
Let $N$ be a submonoid of $M$ and let $A$ be a finite set. 
Suppose that  $\tau \colon A^N \to A^N$ is an injective cellular automaton.   
By Proposition~\ref{p:ind-res}.(ii), the induced cellular automaton $\tau^M \colon A^M \to A^M$ is also injective. Since $M$ is surjunctive, this implies that  $\tau^M$ is  surjective.
Restricting back to $N$, it then follows from  Proposition~\ref{p:ind-rest-inv} and Proposition~\ref{p:ind-res}.(i)  that $\tau = (\tau^M)_N$ is also surjective.
This shows that $N$ is surjunctive.
\end{proof}

\begin{corollary}
\label{c:canc-commutative-monoid-surj}
Every cancellative commutative monoid is surjunctive.
In particular, the additive monoid $\N$ is surjunctive.
\end{corollary}

\begin{proof}
Every cancellative commutative monoid embeds in an abelian group
via the Grothendieck construction
 and it is known (see e.g. \cite[Corollary~3.3.7]{livre}) that every abelian group is surjunctive.
\end{proof}

Let us  recall the definition of amenability for monoids (see e.g.~\cite{day-amenable-semigroups}, \cite{paterson}).
Let $M$ be a monoid and let    $\ell^\infty(M)$ denote the vector space   of all bounded real-valued maps 
$f \colon M \to \R$.
A \emph{mean} on $M$ is an $\R$-linear map $\mu \colon \ell^\infty(M) \to \R$ such that 
$\inf_{m \in M} f(m) \leq \mu(f) \leq \sup_{m \in M} f(m)$ for all $f \in \ell^\infty(M)$.
One says that a mean $\mu$ on $M$   is   \emph{left-invariant} (resp.~\emph{right-invariant}, resp.~\emph{invariant}) if it satisfies
$\mu(f \circ L_m) = \mu(f)$ (resp.~$\mu(f \circ R_m) = \mu(f)$, resp.~$\mu(f \circ L_m) = \mu(f \circ R_m) = \mu(f)$) for all
$f \in \ell^\infty(M)$ and $m \in M$.
The monoid $M$ is called
\emph{left-amenable} (resp.~\emph{right-amenable}, resp.~\emph{amenable}) if it admits a left-invariant (resp.~right-invariant, resp.~invariant)  mean.
As every commutative monoid is amenable (cf. \cite[(H) Section 4]{day-amenable-semigroups}), 
the following result extends Corollary~\ref{c:canc-commutative-monoid-surj}
(compare~\cite[Corollary 1.2.]{semieden}). 

\begin{corollary}
\label{c:amenable-are-surj}
Every cancellative one-sided amenable monoid is surjunctive.
\end{corollary}

\begin{proof} 
It is known \cite[Corollary~3.6]{wilde-witz} that every cancellative one-sided amenable monoid embeds  in an amenable group. 
On the other hand, every amenable group is surjunctive
(see \cite[Theorem~5.9.1]{livre}).
\end{proof}

\begin{remark}
The bicyclic monoid is known to be amenable (see e.g. \cite[Example 2, page 311]{duncan-namioka}).
As the bicyclic monoid is non-surjunctive by Theorem~\ref{t:contains-bicyclic-is-non-surj},
this shows that Corollary~\ref{c:amenable-are-surj} is false  without the cancellability hypothesis.
\end{remark}

\begin{corollary}
Every free monoid is surjunctive.
\end{corollary}

\begin{proof}
If $M$ is the free monoid based on a set $X$, then $M$ naturally embeds in the free group based on $X$.
On the other hand, free groups are known to be surjunctive
(cf. \cite[Corollary~3.3.5]{livre}).
\end{proof}

\begin{corollary}
Let $M$ be a non-surjunctive monoid that is isomorphic to its opposite
 monoid (e.g. $M$ is the bicyclic monoid or $M$ is a non-surjunctive group) and $A$  a set with more than one element. 
Then the monoid $\CA(M;A)$ is non-surjunctive.
\end{corollary}

\begin{proof}
For each $m \in M$, the map $\tau_m \colon A^M \to A^M$, defined by $\tau_m(x) = x \circ L_m$ for all $x \in A^M$, is a cellular automaton
(cf. Example~\ref{ex}.(c)).
Consider the map $\varphi \colon M \to \CA(M;A)$ defined by $\varphi(m) = \tau_m$ for all $m \in M$.
We clearly have $\varphi(1_M) = \Id_{A^M}$ and $\varphi(m m') = \varphi(m') \varphi(m)$ for all $m,m' \in M$.
Moreover, $\varphi$ is injective since $A$ has more than one element.
As $M$ is isomorphic to its opposite, 
we conclude that the monoid $M$ embeds in the monoid  $\CA(M;A)$.
\end{proof}

 \subsection{Another characterization of surjunctive monoids}

The following result yields a characterization of surjunctive monoids involving
 the bicyclic monoid.

\begin{theorem}
\label{t:non-surj-CA-contient-bicyclique}
Let $M$ be a monoid. 
Then the following conditions are equivalent: 
\begin{enumerate}[{\rm (a)}]
\item 
the monoid $M$ is non-surjunctive;
\item 
there exists a finite set $A$ such that the monoid $\CA(M;A)$ contains a submonoid
  isomorphic to the bicyclic monoid.
\end{enumerate}
\end{theorem}

\begin{proof}
Suppose first that $M$ is non-surjunctive. 
Then we can find a finite set $A$ and a
cellular automaton $\tau \colon A^M \to A^M$ which is injective and non-surjective.
We claim that there exists a cellular automaton $\sigma \colon A^M \to A^M$ such that
\begin{equation}
\label{e:sigma-tau}
\sigma \circ \tau = \Id_{A^M}.
\end{equation} 
Indeed, let $Y := \tau(A^M)$ denote the image of $\tau$. 
Since $A^M$ is a compact Hausdorff space,
$\tau$ is a uniform embedding by virtue of Proposition~\ref{p:inj-hom-unif-emb}. 
 This means that  the map $\phi \colon A^M \to Y$, defined by $\phi(x) := \tau(x)$ for all $x \in A^M$, is a uniform isomorphism.
The uniform continuity of $\phi^{-1}$ implies that
there exists a finite subset $S \subset M$ 
having the following property: 
if two configurations $x_1,x_2 \in A^M$ are such that the configurations $\tau(x_1)$ and $\tau(x_2)$ coincide on $S$, 
then $x_1(1_M) = x_2(1_M)$. 
Fix an arbitrary element  $a_0 \in A$. 
We can then define a map $\nu \colon A^S \to A$  by setting,
for all $u \in A^S$,
\[
\nu(u) := 
\begin{cases} 
x(1_M) & \mbox{if there exists $x \in A^M$ such that $u = (\tau(x))\vert_S$},\\
a_0 & \mbox{otherwise}.
\end{cases} 
\]
Consider the cellular automaton $\sigma \colon A^M \to A^M$  with memory set $S$ and local defining map $\nu$.
Then, for all $x \in A^M$ and $m \in M$, we have that
\begin{align*}
\sigma \circ \tau (x)(m) &=
\sigma(\tau(x))(m) \\
&= \nu((m \tau(x))\vert_S) \\ 
&= \nu((\tau(m x))\vert_S) && \text{(since $\tau$ is $M$-equivariant)} \\
&= mx(1_M) && \text{(by definition of $\nu$)} \\
 & =  x(m), 
\end{align*}
so that $\sigma \circ \tau(x) = x$.
This gives us~\eqref{e:sigma-tau}.
\par
On the other hand, we have that
\begin{equation}
\label{e:tau-sigma-not-1}
 \tau \circ \sigma  \not= \Id_{A^M}
\end{equation}
since $\tau$ is not surjective. 
It follows from \eqref{e:sigma-tau} and \eqref{e:tau-sigma-not-1} that
$\tau$ is left-invertible but not right-invertible in the monoid $\CA(M;A)$.
Thus, the monoid $\CA(M;A)$ contains a submonoid isomorphic to the bicyclic monoid by Proposition~\ref{p:charact-bicyclic-supmonoid}.
 This proves that (a) implies (b).
 \par
Conversely, suppose that there exist a finite set $A$ such that the monoid $\CA(M;A)$ contains a submonoid isomorphic to the bicyclic monoid.
From the properties of the bicyclic monoid, we deduce that there exist elements $\sigma,\tau \in \CA(M;A)$ 
satisfying~\eqref{e:sigma-tau} and~\eqref{e:tau-sigma-not-1}. 
It follows that $\tau$ is injective but not bijective. 
 Thus, the monoid  $M$ is not surjunctive.
\end{proof}

\subsection{Surjunctivity of residually finite monoids}

\begin{theorem} 
\label{t:res-finite-surj}
Every residually finite monoid is surjunctive.
\end{theorem} 

\begin{proof}
Let $M$ be a residually finite monoid and $A$ a finite set. 
Suppose that  $\tau \colon A^M \to A^M$ is an injective cellular automaton.
Let us show that $\tau$ is surjective.
\par
By Proposition~\ref{p:per-union-e-gamma}, the set $\Per(A^M)$ of periodic configurations in $A^M$ satisfies
\begin{equation}
\label{e:per-config-union-inv}
\Per(A^M) = \bigcup_{\gamma \in \CC_f(M)} \Inv(\gamma),
\end{equation}
where $\CC_f(M)$ denotes the set of finite index congruence relations on $M$.
\par
Let $\gamma \in \CC_f(M)$. 
Then the set $\Inv(\gamma)$ consists of all configurations $x \in A^M$ that are constant on each equivalence class of $\gamma$
(see Example~\ref{x:inv-gamma-shift}).
As $\gamma$ has only finitely many equivalence classes  and $A$ is finite,
we deduce that $\Inv(\gamma)$ is finite.
On the other hand,
since $\tau$ is $M$-equivariant, it follows from Proposition~\ref{p:E-periodic}.(iii)  that   
$\tau(\Inv(\gamma)) \subset \Inv(\gamma)$.
 As $\tau$ is injective and every injective map from a finite set into itself is surjective, 
we deduce that $\tau(\Inv(\gamma)) = \Inv(\gamma)$.
 \par 
Therefore, by applying $\tau$ to both sides of~\eqref{e:per-config-union-inv}, we  get
\begin{align*}
\tau(\Per(A^M)) 
&= \tau\left(\bigcup_{\gamma \in \CC_f(M)} \Inv(\gamma)\right) \\
&= \bigcup_{\gamma \in \CC_f(M)} \tau(\Inv(\gamma) ) \\
&= \bigcup_{\gamma \in \CC_f(M)} \Inv(\gamma) \\
&= \Per(A^M),
\end{align*} 
 so that
$$
\Per(A^M) = \tau(\Per(A^M)) \subset \tau(A^M) \subset A^M.
$$ 
As $\Per(A^M)$ is dense in $A^M$ by Proposition~\ref{p:res-finite-dynam} 
and $\tau(A^M)$ is closed in $A^M$ since $\tau$ is continuous and $A^M$ is a compact  Hausdorff space,  
we conclude  that $\tau(A^M) = A^M$. Therefore $\tau$ is surjective.
\end{proof}

We shall give an alternative proof of Theorem~\ref{t:res-finite-surj} in Subsection~\ref{subsec:closed-surjun}.

\begin{corollary}
\label{c:fg-comm-surj}
Every finitely generated commutative monoid is surjunctive.
\end{corollary}

\begin{proof}
By a result of Mal'cev~\cite{malcev} (see also~\cite{lallement} and~\cite{carlisle}), 
every finitely generated  commutative monoid is  residually finite.  
\end{proof}

The multiplicative monoid $\N$ is commutative but 
neither finitely generated (since there are infinitely many primes) nor cancellative (since $0$ is absorbing).
Therefore, it  satisfies neither the hypotheses of Corollary~\ref{c:fg-comm-surj} nor those of Corollary~\ref{c:canc-commutative-monoid-surj}.
However, it is residually finite. 
Indeed, if $n_1,n_2$ are distinct elements in $\N$ and $n$ is  an integer such that $n> \max(n_1,n_2)$, then reduction modulo $n$ yields a monoid morphism
$\varphi \colon (\N,\times) \to (\Z/n\Z,\times)$ such that $\varphi(n_1) \not= \varphi(n_2)$.
Therefoer, we also get  the following result:

\begin{corollary}
The multiplicative monoid $\N$ is surjunctive.
\qed
\end{corollary}

A monoid $M$ is called \emph{linear} 
if there exist an integer $n \geq 1$ and a field $K$ such that $M$ embeds in the multiplicative monoid of $n \times n$ matrices with coefficients in $K$.

\begin{corollary}
Every finitely generated linear monoid is surjunctive.
\end{corollary}

\begin{proof}
By a result of Mal'cev \cite{malcev-representation},every finitely generated linear monoid is residually finite. 
 \end{proof}

 % SECTION 6
\section{Marked monoids}
\label{sec:marked-monoids}

\subsection{The space of marked monoids}
Let $M$ be a monoid.

An $M$-quotient is a pair $(\overline{M}, \pi)$, where $\overline{M}$ is a monoid and 
$\pi \colon M \to \overline{M}$ is a surjective monoid morphism. 
We define an equivalence relation on the class of all $M$-quotients by declaring that two $M$-quotients  $(\overline{M}_1,\pi_1)$ and $(\overline{M}_2,\pi_2)$ are 
\emph{equivalent} when there exists a monoid isomorphism $\varphi \colon \overline{M}_2 \to \overline{M}_1$ such that the following diagram is commutative:

$$
\xymatrix{ 
& \overline{M}_2 \ar[dd]^{\varphi \cong} \\
M \ar@{->>}[ru]^{\pi_2} \ar@{->>}[rd]_{\pi_1} & \\
& \overline{M}_1
} 
$$
that is, such that $\pi_1 = \varphi \circ \pi_2$.
\par
An equivalence class of $M$-quotients is called an $M$-\emph{marked monoid}.
Observe that two $M$-quotients $(\overline{M}_1,\pi_1)$ and $(\overline{M}_2,\pi_2)$ are equivalent if and only if 
$\gamma_{\pi_1} = \gamma_{\pi_2}$.
Thus, the set of $M$-marked monoids may be identified with the set $\CC(M)$ consisting of 
all congruence relations on $M$.
 Let us identify the set $\PP(M \times M)$ consisting of all subsets of $M \times M$ with the set $\{0,1\}^{M \times M}$ by means of the bijection from $\PP(M \times M)$ onto $\{0,1\}^{M \times M}$ given by $A \mapsto \chi_A$, where $\chi_A \colon M \times M \to \{0,1\}$ is the characteristic map of $A \subset M \times M$. 
We equip the set $\PP(M \times M) = \{0,1\}^{M \times M}$ with its prodiscrete uniform structure 
and its subset $\CC(M)$ 
 with the induced uniform structure.
Thus, a base of entourages of $\CC(M)$ is provided by the sets
$$
V_F = \{(\gamma_1,\gamma_2) \in \CC(M) \times \CC(M) : \gamma_1 \cap F = \gamma_2 \cap F \},
$$
where $F$ runs over all finite subsets of $M \times M$.
Intuitively, two congruence relations on $M$ are ``close" in $\CC(M)$ when their intersections with a large finite subset of $M \times M$ coincide.

\begin{proposition}
\label{p:properties-space-n-gamma}
Let $M$ be a monoid.
Then the space $\CC(M)$
 of $M$-marked monoids
 is a totally disconnected compact Hausdorff   space.   
\end{proposition}

\begin{proof}
The space $\PP(M \times M)$ is totally disconnected, compact, and Hausdorff since it is a product of finite discrete spaces.
Thus, it suffices to show that $\CC(M)$
   is closed in $\PP(M \times M)$.
To see this, observe that a subset $E \in \PP(M \times M)$ is a congruence relation on $M$ if and only if it satisfies
\begin{enumerate}[(1)]
\item
$(m,m) \in E$ for all $m \in M$;
\item
$(m',m) \in E$ for all $(m,m') \in E$;
\item
$(m,m'') \in E$ for all $(m,m')$ and $(m',m'') \in E$;
\item
$(m_1m_2,m'_1m'_2) \in E$ for all $(m_1,m'_1), (m_2,m'_2) \in E$.
\end{enumerate}
 Denoting, for each $(m,m') \in M \times M$, by $\pi_{m,m'} \colon \PP(M \times M) = \{0,1\}^{M \times M} \to \{0,1\}$
the projection map corresponding to the $(m,m')$-factor, these conditions are equivalent to
\begin{enumerate}[{\rm (1')}]
\item
$\pi_{m,m}(E) = 1$ for all $m \in M$;
\item
$\pi_{m,m'}(E)(\pi_{m',m}(E) - 1) = 0$ for all $m,m' \in M$;
\item
$\pi_{m,m'}(E)\pi_{m',m''}(E)(\pi_{m,m''}(E) - 1) = 0$ for all $m,m',m'' \in M$;
\item
$\pi_{m_1,m'_1}(E)\pi_{m_2,m'_2}(E)(\pi_{m_1m_2,m'_1m'_2}(E) - 1) = 0$ for all $m_1,m_2,m'_1,m'_2 \in M$.
\begin{comment}
\item
$\pi_{mm',mm''}(E)(\pi_{m',m''}(E) - 1) = 0$ for all $m,m',m'' \in M$.
\end{comment}
\end{enumerate}  
As all projection maps are continuous on $\PP(M \times M)$, this shows that $\CC(M)$
 is an intersection of closed subsets of $\PP(M \times M)$
and hence  closed in $\PP(M \times M)$.
\end{proof}

\begin{remark}
If $M$ is countable then the uniform structure on $\CC(M)$ is metrizable.
Indeed, the uniform structure on $\PP(M \times M) = \{0,1\}^{M \times M}$ is metrizable when $M$ is countable. 
More precisely, when $M$ is finite then   $\PP(M)$ is a finite discrete space, while if $M$ is countably-infinite 
then $\PP(M \times M)$
  is homeomorphic to the Cantor set.
\end{remark}

\subsection{Marked monoids and residual properties} 
 
 One says that a property $\PP$ of monoids is \emph{closed under taking monoids} if every submonoid of a monoid satisfying $\PP$ 
 satisfies $\PP$.

\begin{proposition}
\label{p:charcterisation-residual-C}
Let $M$ be a monoid.
Suppose that $\PP$ is a property of monoids that is closed under taking finite direct products and submonoids
(e.g. the property of being finite or the property of being both cancellative and commutative). 
Then the following conditions are equivalent:
\begin{enumerate}[\rm (a)]
\item 
$M$ is residually $\PP$;
\item 
there exists a net $(\gamma_i)_{i \in I}$ in $\CC(M)$ converging to the diagonal $\Delta_M$  
such that the quotient monoids $M/\gamma_i$ satisfy $\PP$ for all $i \in I$.
\end{enumerate}
\end{proposition}

\begin{proof} 
Suppose first that $M$ is residually $\PP$. 
Then, for every pair $d = (m,m') \in M \times M \setminus \Delta_M$, we can find a monoid $N_d$ satisfying $\PP$ and a monoid  morphism $\phi_d \colon M \to N_d$ such that 
$\phi_d(m) \neq \phi_d(m')$. 
 Let $I$ denote the directed set consisting of all finite subsets of $M \times M \setminus \Delta_M$ partially ordered by inclusion. 
Consider, for each $i \in I$, the product morphism 
$$
\psi_i := \prod_{d \in i} \phi_d \colon M \to \prod_{d \in i} N_d
$$   
and let $\gamma_i := \gamma_{\phi_i}$ denote the associated congruence relation
 (cf.~\eqref{e:cong-rel}).
 Observe that every $(m,m') \in i$ satisfies $\psi_i(m) \not= \psi_i(m')$ so that 
 \begin{equation}
 \label{e:i-cap-gamm-i}
 i \cap \gamma_i = \varnothing.
 \end{equation} 
\par
Let $F \subset M \times M$ be a finite subset.
Then $i_0 := F \setminus \Delta_M \in I$.
Moreover, it follows from~\eqref{e:i-cap-gamm-i} that every  $i \in I$ such that $i_0 \subset i$ satisfies  $\gamma_i \cap F = \Delta_M \cap F$.
Consequently, the net $(\gamma_i)_{i \in I}$  converges to $\Delta_M$ in $\CC(M)$. 
\par
Finally, by our assumptions on $\PP$, we have that
$\prod_{d \in i} N_d$ and its submonoid $\psi_i(M) \cong M/\gamma_i$ satisfy $\PP$ for each $i \in I$. 
This shows that (a) implies (b).
\par
Conversely, suppose that (b) holds. Let $m,m' \in M$ be distinct elements and consider the finite set $F = \{(m,m')\} \subset (M \times M) \setminus \Delta_M$. 
Then there exists $i_F \in I$ such that $\gamma_{i_F} \cap F = \Delta_M \cap F$ and $M/\gamma_{i_F}$ satisfy $\PP$. 
As $\Delta_M \cap F = \varnothing$ we have that $(m,m') \notin \gamma_{i_F}$. 
In other words, if $\phi_F \colon M \to M/\gamma_{i_F}$ denotes the canonical epimorphism, 
we have $\phi_F(m) \neq \phi_F(m')$.
We deduce that $M$ is residually $\PP$. This shows (b) $\Rightarrow$ (a).
\end{proof}

Note that in the above proposition, the assumptions on $\PP$ are not needed for the implication (b) $\Rightarrow$ (a).

 \subsection{Marked monoids and invariant subsets}
 Let $M$ be a monoid and let $A$ be a set.
If $\gamma \in \CC(M)$, we have seen in Example~\ref{ex:prodiscrete-top}
that $\Inv(\gamma)$ is the invariant subset of $A^M$ consisting of all configurations $x \in A^M$ that are constant on each equivalence class of $\gamma$. 
  
\begin{theorem}
\label{t:psi-uc-n-p}
Let $M$ be a monoid and let $A$ be a set.
 Equip $\CC(M)$ with the uniform structure induced by the prodiscrete uniform structure on $\PP(M \times M) = \{0,1\}^{M \times M}$ and  
 $\PP(A^{M})$ with the Hausdorff-Bourbaki uniform structure associated with the prodiscrete uniform structure on $A^{M}$. 
Then the map $\Psi \colon \CC(M) \to \PP(A^{M})$ defined by $\Psi(\gamma) := \Inv(\gamma)$ is uniformly continuous.
Moreover, if $A$ contains more than one element then $\Psi$ is a uniform embedding.
\end{theorem}

\begin{proof}
Let  $W$ be an entourage of $\PP(A^{M})$.
Let us show that there exists an entourage $V$ of $\CC(M)$ such that
\begin{equation}
\label{e:unif-cont-psi-mg}
(\gamma_1,\gamma_2) \in V \Rightarrow (\Inv(\gamma_1),\Inv(\gamma_2))   \in W.
\end{equation}
This will prove that $\Psi$ is uniformly continuous.
\par
By definition of the Hausdorff-Bourbaki uniform structure on $\PP(A^M)$, 
there is an entourage $T$ of $A^M$ such that
\begin{equation}
\label{e:ent-in-gr-haus-1}
\widehat{T} := \{(X,Y) \in \PP(A^M) \times \PP(A^M) : Y \subset T[X] \text{ and } X \subset T[Y] \} \subset W.
\end{equation}

Since $A^M$ is endowed with its prodiscrete uniform structure,
there is a finite subset $F \subset M$ such that
\begin{equation}
\label{e:u-f-t-entour}
U := \{(x,y) \in A^M \times A^M : x\vert_F = y\vert_F\} \subset T.
\end{equation}
We claim that the entourage $V$ of $\CC(M)$ given by
$$
V := \{(\gamma_1,\gamma_2) \in \CC(M) \times \CC(M) : \gamma_1 \cap (F \times F) = \gamma_2 \cap (F \times F)\}
$$
satisfies~\eqref{e:unif-cont-psi-mg}. 
\par
To prove our claim, 
suppose that $(\gamma_1,\gamma_2)  \in V$. Let $x \in \Inv(\gamma_1)$. 
Let also $m,m' \in F$ and suppose that
\begin{equation*}
\label{e:mm'-rho}
(m,m') \in \gamma_2.
\end{equation*} 
The fact that $\gamma_1 \cap (F \times F) = \gamma_2 \cap  (F \times F)$ implies that 
\begin{equation}
\label{e:mm'-rho-0}
(m,m') \in \gamma_1.
\end{equation} 
Since $x \in \Inv(\gamma_1)$, 
we deduce from~\eqref{e:mm'-rho-0}  that $x(m) = x(m')$.
Thus, the configuration $x$ is constant on the intersection with $F$ of each $\gamma_2$-class.
It follows that there exists $y \in \Inv(\gamma_2)$ such that $x$ and $y$ coincide on $F$, i.e.,
$(x,y) \in U$.
Since $U \subset T$ by~\eqref{e:u-f-t-entour},
this shows that $\Inv(\gamma_1) \subset T[\Inv(\gamma_2)]$.
Similarly, we get $\Inv(\gamma_1) \subset T[\Inv(\gamma_2)]$,
so that $(\Inv(\gamma_1), \Inv(\gamma_2)) \in \widehat{T}$ by~\eqref{e:ent-in-gr-haus-1}.
As $\widehat{T} \subset W$ by~\eqref{e:ent-in-gr-haus-1},
we conclude that $V$ satisfies~\eqref{e:unif-cont-psi-mg}.
 Therefore $\Psi$ is uniformly continuous.  
\par
As $\CC(M)$ is compact by Proposition \ref{p:properties-space-n-gamma}, we deduce that $\Psi$ is uniformly continuous by applying 
Proposition~\ref{p:cont-comp-implies-unifcont}. 
\par
Suppose now that $A$ has more than one element and let us show that $\Psi$ is injective.
Let $\gamma_1, \gamma_2 \in \CC(M)$ and suppose that $\Psi(\gamma_1) = \Psi(\gamma_2)$, that is, $\Inv(\gamma_1) = \Inv(\gamma_2)$.
Let $m,m' \in M$ such that  $(m,m') \notin \gamma_1$.
Fix $a,b \in A$ with $a \not= b$ and consider the configuration $x \in A^M$ defined, for all $u \in M$,  by 
$x(u) = a$ if $u$ belongs to the $\gamma_1$-class of $m$ and $x(u) = b$ otherwise. 
 It is clear that  $x \in \Inv(\gamma_1)$.
 Therefore $x \in \Inv(\gamma_2)$.
 This implies $(m,m') \notin \gamma_2$ since $x(m) \not= x(m')$.
We deduce that  $\gamma_2 \subset \gamma_1$. 
By symmetry, we also have $\gamma_1 \subset \gamma_2$. Therefore $\gamma_1 = \gamma_2$. 
This shows that $\Psi$ is injective.
On the other hand, $\Psi(\gamma) = \Inv(\gamma)$ is closed in $A^M$ for all $\gamma \in \CC(M)$ by Proposition~\ref{p:inv-gamma-closed}.
As $\CC(M)$ is compact and the set of closed subsets of $A^M$ is Hausdorff for the Hausdorff-Bourbaki topology by Proposition~\ref{p:grom-haus-sep-closed}, we conclude that $\Psi$ is a uniform embedding by applying Proposition \ref{p:inj-hom-unif-emb}.     
\end{proof}

\subsection{Cellular automata over quotient monoids}
Let $M$ be a monoid and let  
 $\gamma \in \CC(M)$. 
 Denote by
$\pi_\gamma \colon M \to M/\gamma$  the quotient monoid morphism, i.e., the map sending each $m \in M$ to its $\gamma$-class $[m]$.
Let now $A$ be a set and consider the configuration space $A^M$ (equipped with its prodiscrete topology and with the shift action of $M$). 
We have seen in  Example~\ref{ex:prodiscrete-top}  
that $\Inv(\gamma)$ is the subset of $A^M$ consisting of all configurations $x \in A^M$ that are constant on each  $\gamma$-class.
In other words, the map  
$\pi_\gamma^* \colon A^{M/\gamma} \to \Inv(\gamma)$ defined by
\begin{equation}
\label{e:def-pi-gamma-star}
\pi_\gamma^*(y) := y \circ \pi_\gamma
\end{equation}
 for all $y \in A^{M/\gamma}$,  is bijective.
 We equip the set $A^{M/\gamma}$   with its prodiscrete uniform structure and with the $M/\gamma$-shift.
 On the other hand,  $\Inv(\gamma) \subset A^M$ is equipped with the uniform structure induced by the 
 prodiscrete uniform structure on $A^{M}$.
 Also, since $mx = m'x$ for all $x \in \Inv(\gamma)$ and $(m,m') \in \gamma$, the shift action of $M$ on $A^M$ naturally induces an action of $M/\gamma$ on $\Inv(\gamma)$
 given by
 \begin{equation}
 \label{e:action-ind-M-gamma}
 [m] x := mx
 \end{equation}
 for all $m \in M$ and $x \in \Inv(\gamma)$.

 \begin{proposition}
 \label{p:inv-sim-A-quotient-conf}
The  map $\pi_\gamma^* \colon A^{M/\gamma} \to \Inv(\gamma)$
 defined by~\eqref{e:def-pi-gamma-star} 
 is an $M/\gamma$-equivariant uniform isomorphism.
 \end{proposition}

\begin{proof}
Let $Y \in A^{M/\gamma}$ an $m \in M$.
We have that
$$
[m] y = y \circ R_{[m]}
$$
so that
$$
\pi_\gamma^*([m] y) = y \circ R_{[m]} \circ \pi_\gamma.
$$
Using the fact that  $R_{[m]} \circ \pi_\gamma = \pi_\gamma \circ R_m$ since $\pi_\gamma$ is a monoid morphism, we deduce that 
\begin{align*}
\pi_\gamma^*([m] y) &= y \circ \pi_\gamma \circ R_m \\
&= m(y \circ \pi_\gamma) \\
&= m \pi_\gamma^*(y) \\
&= [m] \pi_\gamma^*(y) && \text{(by \eqref{e:action-ind-M-gamma})}.
\end{align*}
This shows that $\pi_\gamma^*$ is $M/\gamma$-equivariant.
\par
Let $F \subset M$ and $F' := \pi_\gamma(F)$.
If $y_1,y_2 \in A^{M/\gamma}$ coincide on   $F'$, then 
$\pi_\gamma^*(y_1)$ and $\pi_\gamma^*(y_2)$ coincide on $F$.
As $F'$ is finite whenever $F$ is finite, this implies that $\pi_\gamma^*$ is uniformly continuous.
\par
Suppose now that $E$ is a finite subset of $M/\gamma$, that is, a finite set of $\gamma$-classes.
As $\pi_\gamma$ is surjective, we can find a finite subset $\widetilde{E} \subset M$ such that $E = \pi_\gamma(\widetilde{E})$.
If two configurations $x_1,x_2 \in \Inv(\gamma)$ coincide on $\widetilde{E}$, then
$(\pi_\gamma^*)^{-1}(x_1)$ and $(\pi_\gamma^*)^{-1}(x_2)$ coincide on $E$.
This shows that $(\pi_\gamma^*)^{-1}$ is uniformly continuous. 
  \end{proof}

Assume now that  $\tau \colon A^M \to A^M$ is a cellular automaton over $M$.
Since, by definition,  $\tau$ is $M$-equivariant, 
we have that $\tau(\Inv(\gamma)) \subset \Inv(\gamma)$ by
Proposition~\ref{p:E-periodic}.(iii).
 Thus, we can define a map $\tau' \colon A^{M/\gamma} \to A^{M/\gamma}$ by setting
\begin{equation}
\label{e;tau'}
\tau' = (\pi_\gamma^*)^{-1} \circ \tau\vert_{\Inv(\gamma)} \circ \pi_\gamma^*,
\end{equation}
where $\tau\vert_{\Inv(\gamma)} \colon \Inv(\gamma) \to \Inv(\gamma)$ denotes the restriction of $\tau$ to $\Inv(\gamma)$.
In other words, the map $\tau'$ is obtained by conjugating by $\pi_\gamma^*$ the restriction of $\tau$ to $\Inv(\gamma)$, so that the diagram
$$
\begin{CD}
A^{M/\gamma}  @>\pi_\gamma^*>> \Inv(\gamma) \subset A^M \\
@V{\tau'}VV  @VV{\tau\vert_{\Inv(\gamma)}}V \\
A^{M/\gamma} @>>\pi_\gamma^*> \Inv(\gamma)
\end{CD}
$$
is commutative.
\par
Suppose that $S \subset M$ is a memory set for $\tau$ with $\mu \colon A^S \to A$ the associated local defining map.
Consider the finite subset $S' := \pi_{\gamma}(S) \subset M/\gamma$
and the map $\mu' \colon A^{S'} \to A$ defined by $\mu'= \mu \circ \pi'_{\gamma}$, where $\pi'_{\gamma} \colon A^{S'} \to A^S$ is the injective map induced by 
the restriction of $\pi_{\gamma}$ to $S$.

 \begin{proposition}
\label{p:quotient-automaton}
The map $\tau' \colon A^{M/\gamma} \to A^{M/\gamma}$
defined by~\eqref{e;tau'} 
is a cellular automaton over the monoid $M/\gamma$
admitting $S'$ as a memory set and $\mu' \colon A^{S'} \to A$ as the associated local defining map.
 \end{proposition}

\begin{proof}
Let $y \in A^{M/\gamma}$ and $m \in M$.   
We then have
\begin{align*}
\tau'(y)([m]) 
&= \tau(y \circ \pi_\gamma)(m) \\
& = \mu\left((m(y \circ \pi_\gamma))\vert_S\right)\\
& = \mu\left(([m](y \circ \pi_\gamma))\vert_S\right) &&\text{(by \eqref{e:action-ind-M-gamma})}\\
& = \mu'\left(([m]y)\vert_{S'}\right).
\end{align*}
This shows that $\tau'$ is a cellular automaton over $M/\gamma$ admitting $S'$ as a memory set and $\mu'$ as the associated local defining map.
\end{proof}

\begin{remark}
One can give a direct proof of the fact that $\tau'$ is a cellular automaton in the following way.
Since $\tau$ is a cellular automaton over $M$, it is uniformly continuous and $M$-equivariant.
Consequently, $\tau\vert_{\Inv(\gamma)}$ is  uniformly continuous and $M/\gamma$-equivariant.
As $\pi_\gamma^*$ is an $M/\gamma$-equivariant uniform isomorphism by Proposition~\ref{p:inv-sim-A-quotient-conf},
we deduce from~\eqref{e;tau'}   that
$\tau'$ is a composite of uniformly continuous $M/\gamma$-equivariant maps.
 Therefore $\tau'$ is a cellular automaton over $M/\gamma$.
 \end{remark}

Consider now the map $\Phi_\gamma \colon \CA(M;A) \to \CA(M/\gamma;A)$ given by $\Phi(\tau) = 
\tau'$, where $\tau'$ is defined by \eqref{e;tau'}. 
We have the following:

\begin{proposition}
\label{p:phi-gamma-epi}
The map $\Phi_\gamma \colon \CA(M;A) \to \CA(M/\gamma;A)$ is a monoid epimorphism.
\end{proposition}

\begin{proof}
Let $\sigma \colon A^{M/\gamma}\to A^{M/\gamma}$ be a cellular automaton over $M/\gamma$ 
with memory set $S' \subset M/\gamma$ and associated local defining map $\nu \colon A^{S'} \to A$. 
Let $S \subset M$ be a finite set such that $\pi_{\gamma}$ induces by restriction a bijection $\phi \colon S \to S'$. Consider the map $\mu \colon A^{S} \to A$ defined by $\mu(y) = \nu(y \circ \phi^{-1})$ for all $y \in A^{S}$. Let $\tau \colon A^M \to A^M$ be the cellular automaton
over $M$ with memory set $S$ and local defining map $\mu$. 
We have
\[
\mu'(z) = (\mu \circ \pi'_\gamma)(z) = \nu(\pi'_\gamma(z) \circ \phi^{-1}) = \nu(z)
\]
for all $z \in A^{S'}$. 
It follows that $\mu' = \nu$, and $\tau' = \sigma$. This shows that $\Phi_\gamma$ is surjective.
\par
The fact that $\Phi_\gamma$ is a monoid morphism immediately follows from \eqref{e;tau'}.
\end{proof}

\subsection{Gromov's injectivity lemma for monoid actions}

The following result will play a central role in the proof of Theorem~\ref{t:surjunct-closed}.   
 It extends Gromov's injectivity lemma \cite[Lemma~4.H'']{gromov-esav}
 (see also \cite[Theorem~4.1]{expansive} and \cite[Theorem~3.6.1]{livre}).

\begin{theorem} 
\label{t;gil} 
Let $X$ be a uniform space endowed with a uniformly continuous and expansive action of a 
monoid $M$.
Let $f \colon X \to X$ be a uniformly continuous and $M$-equivariant map.
Suppose that $Y$ is a subset of $X$ such that
the restriction of $f$ to $Y$ is a uniform embedding from $Y$ into $X$.
Then there exists an entourage $V$ of $X$ satisfying the following 
property: if $Z$ is an $M$-invariant subset of $X$ such that $Z \subset V[Y]$, then
the restriction of $f$ to $Z$ is injective.
\end{theorem}

(We recall that the notation $Z \subset V[Y]$ means that for each $z \in Z$ there exists $y \in Y$ such that $(z,y) \in V$.)

\begin{proof}
By expansivity of the action of $M$, there is an entourage $W_0$ 
of $X$ such that
\begin{equation}  
\label{e:expansive2}
\bigcap_{m \in M} m^{-1}(W_0) = \Delta_X.
\end{equation}
It follows from the axioms of a uniform structure that we can find a
symmetric entourage $W$ of $X$ such that
\begin{equation} 
\label{e:W-W0}
W \circ W \circ W \subset W_0.
\end{equation}
Since the  restriction of $f$ to $Y$ is a uniform
embedding, we can find an entourage $T$ of $X$ such that
\begin{equation} 
\label{e:T-W}
(f(y_1),f(y_2)) \in T \Rightarrow (y_1,y_2) \in W
\end{equation}
for all $y_1,y_2 \in Y$.
Let $U$ be a symmetric entourage of $X$ such that
\begin{equation} \label{e:U-T}
U \circ U \subset T.
\end{equation}
Since $f$ is uniformly continuous, we can find an entourage $E$ of
$X$ such that
\begin{equation} 
\label{e:E-U}
(x_1,x_2) \in E \Rightarrow (f(x_1),f(x_2)) \in U
\end{equation}
for all $x_1,x_2 \in X$.
\par
Let us show that the entourage $V = W \cap E$ has the required property. So let $Z$ be a 
$M$-invariant subset of $X$  such that 
$Z \subset V[Y]$ and let us show that the restriction of $f$ to 
$Z$ is injective.
\par
Let $z'$ and $z''$ be points in $Z$ such that $f(z') = f(z'')$.
Since $f$ is $M$-equivariant, we have
\begin{equation}
\label{e;z',zz''}
f(m z') = f(m z'')
\end{equation} 
for all $m \in M$. 
As the points $m z'$ and $m z''$ stay in $Z$, the fact that $Z \subset V[Y]$ implies that there are points $y_m'$ and $y_m''$ in $Y$ such that $(m z',y_m') \in V$
and $(m z'',y_m'') \in V$.
Since $V \subset E$, it follows from \eqref{e:E-U} that
$(f(m z'),f(y_m'))$ and $(f(m z''),f(y_m''))$ are both
in $U$. As $U$ is symmetric, we also have $(f(y'_m), f(m z')) \in U$. We deduce that
$(f(y_m'),f(y_m'')) \in U \circ U \subset T$
by using \eqref{e:U-T} and \eqref{e;z',zz''}. 
This implies $(y_m',y_m'') \in W$ by \eqref{e:T-W}. On the other hand, we also have 
$(m z', y_m') \in W$ and $(y_m'', m z'') \in W$ since $V \subset W$ and $W$ is symmetric. 
It follows that
$$
(m z', m z'')  \in W \circ W \circ W \subset W_0
$$
by \eqref{e:W-W0}.
This gives us
$$
(z',z'') \in \bigcap_{m \in M} m^{-1}(W_0),
$$
and hence $z' = z''$ by \eqref{e:expansive2}. Thus the
restriction of $f$ to $Z$ is injective.
\end{proof}

\subsection{Closedness of marked surjunctive monoids}
\label{subsec:closed-surjun}
\begin{theorem}
\label{t:surjunct-closed}
Let $M$ be a monoid. 
Then the set of congruence relations $\gamma$ on $M$ such that the quotient monoid $M/\gamma$ is surjunctive is closed in $\CC(M)$.  
\end{theorem}

\begin{proof}
Let $\gamma \in \CC(M)$ and let $(\gamma_i)_{i \in I}$ be a net in $\CC(M)$ converging to $\gamma$.
Suppose that the monoids $M/\gamma_i$ are surjunctive for all $i \in I$. Let us show that the monoid $M/\gamma$ is also surjunctive.
\par
Let $A$ be a finite set and let $\tau \colon A^{M/\gamma} \to A^{M/\gamma}$ be an injective cellular automaton over the monoid $M/\gamma$ and the alphabet $A$.
By Proposition~\ref{p:phi-gamma-epi}, 
there exists  a cellular automaton
$\widetilde{\tau} \colon A^M \to A^M$ over $M$ 
such that $\tau = \Phi_\gamma(\widetilde{\tau})$.
\par
To simplify notation, let us set $X := A^M$, $f := \widetilde{\tau}$, $Y = \Inv(\gamma)$ and $Z_i = \Inv(\gamma_i)$.
 We claim that the hypotheses of Theorem~\ref{t;gil} are satisfied by $X$, $f$,  $Y$, and the $M$-shift on $X$.
Indeed, we first observe that the action of $M$ on $X$ is uniformly continuous and expansive by Proposition~\ref{p;UCE}.
On the other hand, the map $f \colon X \to X$ is uniformly continuous and 
$M$-equivariant by definition of a cellular automaton.
Moreover, the restriction of $f$ to $Y$ is injective since this restriction is conjugate to $\tau$
by definition of $\Phi_\gamma$.
Finally, we observe that $X$ is Hausdorff and that $Y$ is  closed  in $X$
(by Proposition~\ref{p:inv-gamma-closed})
 and hence compact.
Therefore, it follows from Proposition~\ref{p:inj-hom-unif-emb} that the restriction of  $f$ to $Y$ is a uniform embedding. 
By applying Theorem~\ref{t;gil}, we deduce that there exists an entourage $V$ of $X$ such that if $Z$ is an $M$-invariant subset of $X$ with $Z \subset V[Y]$, then the restriction of $f$ to $Z$ is injective.
\par
Since the net $(Z_i)_{i \in I}$ converges to $Y$ for the Hausdorff-Bourbaki topology on $\PP(X)$ by Theorem \ref{t:psi-uc-n-p}, there is an element $i_0 \in I$ such that $Z_i \subset V[Y]$ for all $i \geq i_0$.
As the sets $Z_i$ are $M$-invariant by Proposition~\ref{p:E-periodic-0}, it follows that the restriction of $f$ to $Z_i$ is injective for all $i \geq i_0$.
On the other hand,   $f(Z_i) \subset Z_i$ and the restriction of $f$ to $Z_i$ is conjugate to a cellular automaton $\tau_i \colon A^{M/\gamma_i} \to A^{M/\gamma_i}$ over the monoid $M /\gamma_i$ for all $i \in I$ by Proposition \ref{p:quotient-automaton}.
As the monoids $M/\gamma_i$ are surjunctive by our hypotheses, we deduce that
$f(Z_i) = Z_i$ for all $i \geq i_0$. 
Now, it follows from Proposition \ref{p:uc-implies-gh-uc} that the net $(f(Z_i))_{i \in I}$ converges to $f(Y)$ in $\PP(X)$.
Thus, the net $(Z_i)_{i \in I}$ converges to both $Y$ and $f(Y)$. As $Y$ and $f(Y)$ are closed in $X$ (by compactness of $Y$), we deduce that $Y = f(Y)$ 
by applying Proposition~\ref{p:grom-haus-sep-closed}.
This shows that $\tau$ is surjective since $\tau$ is conjugate to the restriction of $f$ to $Y$. 
Consequently, the monoid $M/\gamma$ is surjunctive.
\end{proof}

By combining Theorem~\ref{t:surjunct-closed}, Proposition~\ref{p:charcterisation-residual-C}, 
and Proposition~\ref{p:finite-surj}, we recover the fact that all residually finite monoids are surjunctive (Theorem \ref{t:res-finite-surj}). 
Similarly, from Theorem~\ref{t:surjunct-closed}, Proposition~\ref{p:charcterisation-residual-C},
and Corollary~\ref{c:canc-commutative-monoid-surj}, 
we  get the following result.

\begin{corollary}
Let $\PP$ be the property for monoids of being both cancellative and commutative.
Then every residually $\PP$ monoid is surjunctive.
\end{corollary} 

  % SECTION 7
\section{Some open problems}

We have  been unable to answer the following questions.

\begin{enumerate}[(Q1)]
% Q1
\item
Is every commutative monoid surjunctive?
% Q2
\item
Is every locally finite monoid surjunctive?
% Q3
\item
Is every locally surjunctive monoid surjunctive?
% Q4
\item
Is the opposite monoid of every surjunctive monoid surjunctive? 
% Q5
\item
Is every cancellative monoid surjunctive?
% Q6
\item
Is every monoid that is sofic in the sense of \cite{semisofic} surjunctive?
% Q7
\item
Does every non-surjunctive monoid  contain a submonoid isomorphic to the bicyclic monoid?
\end{enumerate}

Note that an affirmative answer to (Q3) would imply an affirmative answer to (Q1) and (Q2), and that an affirmative answer to
(Q7) would imply an affirmative answer to (Q4), (Q5), and (Q6).
The answers to (Q1), (Q2), (Q3),  (Q4), and (Q6)  are known to be  affirmative for groups (see e.g.~\cite{livre}).
 Of course, an affirmative answer to (Q5) would  imply the Gottschalk conjecture (every group is surjunctive).  
 
 \def\cprime{$'$} \def\cprime{$'$} \def\cprime{$'$}

% \bibliographystyle{siam}
%\bibliography{michel_monomarked}

\end{document}